\documentclass[11pt]{amsart}

\usepackage{amssymb}
 \usepackage{graphicx}
 \usepackage{amscd}

 \def\al{\alpha}
 \def\be{\beta}
 \def\de{\delta}
 \def\eps{\varepsilon}
 \def\deta{{\dot{\eta}}}
 \def\ga{\gamma}

 \def\la{\lambda}
 \def\si{\sigma}
 
 \def\om{\omega}

 \def\deta{{\dot{\eta}}}

 \def\R{{\mathbb R}}
 
 \def\N{{\mathbb N}}

 \def\ov{\overline}

\newcommand{\DD}[1]{\mbox{\rm #1}}
\newcommand{\dd}{\DD{d}}

 \def\A{{\mathcal A}}

 \def\cH{{\check{\mathbf H}}}
 \def\cH_i{{\check{ H_i}}}
 
 \def\cH{{\check{ L_i}}}

 \DeclareMathOperator{\interior}{int}
 
 \DeclareMathOperator{\co}{co}
 \DeclareMathOperator{\OO}{O}

  \renewcommand{\proofname}{{\bf Proof:}}

 \theoremstyle{plain}

 \newtheorem{Thm}{Theorem}[section]
 
 \newtheorem{Lemma}[Thm]{\bf Lemma}
 
 \newtheorem{Corollary}[Thm]{\bf Corollary}
 \newtheorem{Theorem}[Thm]{\bf Theorem}
 \newtheorem{Proposition}[Thm]{\bf Proposition}

 \theoremstyle{definition}

 \theoremstyle{remark}
 \newtheorem{Remark}[Thm]{\bf Remark}

 \newtheoremstyle{Cl}
  {5pt}
  {3pt}
  {\sl}
  {}
  {\it}
  {:}
  {.5em}
  {}

 \theoremstyle{Cl}

 \def\begincproof{
                  \renewcommand{\proofname}{\it Proof:}
                  \begin{proof}
                 }

 \def\endcproof{
                \renewcommand{\qedsymbol}{$\diamondsuit$}
                \end{proof}
                \renewcommand{\qedsymbol}{\openbox}
                \renewcommand{\proofname}{\bf Proof:}
               }

 \renewcommand{\proofname}{{\bf Proof:}}

 \title
 {Singularly perturbed control systems with noncompact fast variable}

\thanks{}

\author[Nguyen]{Thuong Nguyen }

 \address{Department of Mathematics\\ Quy Nhon University\\170
 An Duong Vuong street\\ Quy Nhon city\\Vietnam}
\email{nguyenngocquocthuong@qnu.edu.vn}

 \author[Siconolfi]{Antonio Siconolfi}

\address{Dipartimento di Matematica\\
                 Universit\`a degli Studi di Roma ``La Sapienza''\\
                  00185 Roma\\
                   Italy.}
\email{siconolf@mat.uniroma1.it}

\begin{document}

   \begin{abstract} We deal with  a singularly perturbed optimal control problem with  slow
  and fast variable depending on a parameter $\eps$. We study the asymptotic, as $\eps$ goes to $0$, of
the corresponding value functions, and show  convergence, in the
sense of weak semilimits, to  sub and supersolution of a suitable
limit equation containing  the effective Hamiltonian.

The novelty  of our contribution is that no compactness condition
are assumed on the fast variable. This generalization requires, in
order to perform the asymptotic procedure,  an accurate qualitative
analysis of some auxiliary equations posed on the space of fast
variable. The task is accomplished  using some tools of Weak KAM
theory, and in particular the notion of Aubry set.

\end{abstract}


 \maketitle

\section{Introduction}

  \parskip +3pt

We study a singularly perturbed optimal control problem with a slow
variable, say $x$,  and a fast one, denoted by $y$, with
 dynamics depending on a parameter $\eps$ devoted to become
infinitesimal. We are interested in  the asymptotic, as $\eps$ goes
to $0$, of the corresponding value functions $V^\eps$, depending on
slow, fast variable and time,  in view of proving convergence, in
the sense of weak semilimits, to some functions independent of $y$,
related to a limit control problem where $y$ does not appear any
more, at least as state variable.

More precisely, we exploit that the $V^\eps$ are solutions, in the
viscosity sense, to a time--dependent Hamilton--Jacobi--Bellman
equation of the form
\[ u^\eps_t + H \left (x,  y , D_x u^\eps, \frac{D_y u^\eps} \eps \right ) =  0\]
and show that the upper/lower  weak semilimit   is sub/supersolution
to a limit equation
\[u_t + \ov H(x,Du)=0\]
containing the so--called effective Hamiltonian $\ov H$, obtained
via a canonical procedure we describe below from the Hamiltonian of
the approximating equations. We also show that initial conditions,
i.e. terminal costs,  are transferred, with suitable adaptations, to
the limit. See Theorems \ref{mainuno}, \ref{maindue}, which are the
main results of the paper.

We tackle the subject through a   PDE approach first proposed in
this context by Alvarez--Bardi, see \cite{AB1}, \cite {AB2} and the
survey booklet \cite{AB3},  in turn inspired by techniques developed
in the framework of homogenization of Hamilton--Jacobi equations by
Lions--Papanicolau--Varadhan and Evans, see \cite{LPV}, \cite{E1},
\cite{E2}. The singular perturbation can be actually viewed as a
relative homogenization of slow with respect fast variable. In the
original formulation, homogenization was obtained assuming
periodicity  in the underlying space plus coercivity of the
Hamiltonian in the momentum variable.

Alvarez--Bardi keep periodicity in $y$, but do without coercivity,
and assume instead bounded time controllability in the fast
variable. A condition of this kind is indeed unavoidable, otherwise
it cannot be expected to get rid of $y$ at the limit,  or even to
get any limit. Another noncoercive homogenization problem, arising
from turbulent combustion models, has been   recently investigated
with similar techniques in \cite{XY}.

The novelty of our contribution is that we remove any compactness
condition on the fast variable, and this  requires major adaptations
in perturbed test function method, which is the core of the
asymptotic procedure. We further comment on it later on.

Following a more classical control--theoretic approach, namely
directly working on the trajectory of the dynamics,
Arstein--Gaitsgory, see \cite{AG} and \cite{A1}, \cite{A2}, have
studied a similar model  replacing in a sense periodicity  by a
coercivity condition  in the cost, and allowing $y$ to vary in the
whole of $\R^M$, for some dimension $M$. Beside proving convergence,
they also provide a thorough description of the limit control
problem, in terms of occupational measures,  see \cite{A2}. This is
clearly a relevant aspect of the topic, but we do not treat it here.

Our aim is to recover their  results adapting Alvarez--Bardi
techniques. We assume, as in \cite{AG} and \cite{A1}, coercivity of
running cost, see {\bf (H4)}, and  a controllability condition, see
{\bf (H3)}, stronger than the one used in \cite{AB1}, \cite {AB2},
\cite{AB3} and implying, see Lemma \ref{astima}, coercivity of the
corresponding  Hamiltonian, at least in the fast variable. We do
believe that our methods can also work under bounded time
controllability, and so without any coercivity on $H$, but this
requires more work, and  the details have still to be fully checked
and written down.

The focus of  our analysis is on the associate cell problem, namely
 the one--parameter family of
stationary equations, posed in the space of fast variable, obtained
by freezing in $H$ slow variable and momentum, say at a value
$(x_0,p_0)$. Its role, at least in the periodic case, is twofold: it
provides a definition of the effective Hamiltonian $\ov H$ at
$(x_0,p_0)$ as the minimum value of the parameter for which there is
a subsolution (then also supersolutions or solutions do exist),  the
corresponding equation will be  called critical in what follows, and
critical sub/supersolutions   play the crucial role of correctors in
the perturbed test function method.

The absence of compactness calls into questions  the very status of
the critical value $\ov H(x_0,p_0)$ since, in contrast to what
happens when periodicity is assumed,  the existence of solutions
does not characterize any more the critical  equation, see Appendix
\ref{appendice}. Moreover  critical sub/supersolutions must enjoy
suitable additional properties, as explained below, to be effective
in the asymptotic procedure.

The two issues are intertwined. By performing a rather accurate
qualitative analysis of the cell problems, we show that (sub/super)
solutions usable  as correctors can be obtained only for the
critical equation. We make essentially use for that  of  tools
issued from weak KAM theory, and in particular of the capital notion
of Aubry set. As far as we know, it is the first time that this
methodology finds a specific application in singular perturbation or
homogenization problems.

The geometric counterpart of coercivity in the cost functional is
that the critical equation  has a nonempty compact Aubry set for
every fixed $(x_0,p_0)$,  see Lemma \ref{lemHkdue}, which in turn
implies existence of coercive  solutions possessing a simple
representation formula in  terms of a related intrinsic metric,
 and bounded subsolutions as well, see Propositions \ref{Hkuno}, \ref{Hkdue}. Coercive
solutions, up to  modification  depending on $\eps$ (see Subsection
\ref{construct}), are used in the upper semilimit part of the
asymptotic, which is the most demanding point of the analysis.

The paper is organized as follows. In Section 2 we give some
preliminary material and standing assumptions, we then study some
relevant property of controlled dynamics and how they affect value
functions. Approximating Hamilton--Jacobi--Bellman equations and
limit problem are also defined. Section 3 is about cell problems and
construction of distinguished critical sub/supersolutions  to be
used as correctors. Sections 4 contain the main results. The
appendix is devoted to review some basic facts of metric approach
and Weak KAM theory for general Hamilton--Jacobi equations.

  \bigskip

\section{Setting of the problem}\label{setting}

\parskip +3pt

  \subsection{Notations and terminology} \label{notations}
Given an Euclidean space, say to fix ideas $\R^N$, for some $N \in
\N$, $x \in \R^N$ and $R >0$ we denote by
 $B(x,R)$ the open ball centered at $x$ with radius $R$. Given $B \subset \R^N$, we indicate
  by $\ov B$, $\interior B$, its closure and interior, respectively.
  Given subsets $B$, $C$, and a scalar $\la$, we set
  \begin{eqnarray*}
    B + C  &=& \{x+y \mid x \in B, \, y \in C\} \\
    \la \, B &=& \{ \la \, x \mid x \in B\}.
  \end{eqnarray*}

\smallskip

We make precise that in all Hamilton--Jacobi equations we will
consider throughout  the paper the term (sub/super) solution must be
understood in the viscosity sense.

Given an upper semicontinuous (resp. lower semicontinuous) $u: \R^N
\to \R$, we say that a function $\psi$ is supertangent (resp.
subtangent) to a $u$ at some point $x_0$ if it is of class $C^1$,
$u=v$ at $x_0$  and \[ \psi \geq u \;\; \hbox{(resp. $\psi \leq u$),
\quad locally at $x_0$.}\] If strict inequalities hold in the above
formula then $\psi$ will be called strict supertangent (resp.
subtangent).

Given a sequence of locally equibounded functions $u_n: \R^M \to
\R$, the upper weak semilimit (resp lower weak semilimit) is defined
via the formula
\begin{eqnarray*}
({\limsup}^{\#} u_n)(x)&=&\sup\{\limsup_n u_n(x_n) \mid x_n\to x\}\\
(\hbox{resp.} \quad ({\liminf}_{\#} u_n)(x)&=& \inf\{\liminf_n
u_n(x_n) \mid x_n\rightarrow x\}).
\end{eqnarray*}
If $u$ is a locally bounded function and we take in the above
formula the sequence $u_n$ constantly equal to $u$ then we get
through upper (resp. lower) weak semilimit the upper (resp. lower)
semicontinuous envelope of $u$, denoted by $u^\#$ (resp. $u_\#$).
It is minimal (resp. maximal) upper (resp. lower) semicontinuous
function greater (resp. less) than or equal to $u$.

\medskip

\subsection{Assumptions} We assume that the slow variable, usually denoted by $x$,
lives in $\R^N$ and the fast variable $y$ in $\R^M$, for given
positive integers $N$, $M$. We denote by $A$ the control set, by
$f:\R^N \times \R^M \times A \to \R^N$, $g:\R^N \times \R^M \times A
\to \R^M$ the controlled vector fields related to slow and fast
dynamics, respectively. We also have a running cost $\ell: \R^N
\times \R^M \times A \to \R$ and a terminal cost $u_0: \R^N \times
\R^M \to \R$.  We call, as usual, control  a measurable trajectory
defined in $[0,+ \infty)$ taking values in $A$.  We require:

\smallskip

\begin{itemize}
    \item[{\bf(H1)}] \underline{Control set}: $A$ is  a compact subset of some Euclidean
    space;
    \smallskip
    \item[{\bf(H2)}]   \underline{Controlled dynamics}: There is a constant   $L_0 > 0$ with
    \begin{eqnarray*}
     |f(x_1,y_1,a) -f(x_2,y_2,a)| &\leq& L_0 \, (|x_1-x_2| +
    |y_1-y_2|) \\
      |g(x_1,y_1,a) -g(x_2,y_2,a)| &\leq& L_0 \, (|x_1-x_2| +
    |y_1-y_2|)
    \end{eqnarray*}
 for any $(x_i,y_i)$, $i=1,2$ in $\R^N \times \R^M$ and $a \in A$;
 we assume in addition that $|f|$ is bounded with upper bound
 denoted by $Q_0$;
    \item[{\bf(H3)}] \underline{Total controllability}:
For any compact set $K \subset \R^N \times \R^M$ there exists
$r=r(K) >0$ such that
\[B(0,r) \subset \ov\co \, g(x,y,A) \qquad\hbox{for $(x,y) \in K$,}\]
where $g(x,y,A)= \{g(x,y,a) \mid a \in A\}$;
\smallskip
\item[{\bf(H4)}] \underline{Running cost}:
 $\ell$ is continuous in
$\R^N \times \R^M \times A$, and for any compact set  $B \subset
\R^N $
\begin{equation}\label{coe}
    \lim_{|y| \to + \infty} \,  \min_{(x,a)\in B \times A}  \ell(x,y,a) = +
    \infty;
\end{equation}
    \item[{\bf(H5)}] \underline{Terminal cost}: $u_0$ is  continuous and bounded from below in
     $\R^N \times \R^M$. To simplify notations, $-Q_0$, see {\bf (H2}), is also taken
as lower bound of $u_0$ in $\R^N \times \R^M$.
\end{itemize}

\medskip

Taking into account Assumption {\bf(H5)},  we define

\begin{equation}\label{initial}
    \ov u_0(x) = \inf_{y \in \R^M} u_0(x,y) \qquad\hbox{for any $x
    \in \R^N$.}
\end{equation}

This function is apparently  upper semicontinuous,  and will play
the role of initial condition in the limit  equation we get in the
asymptotic procedure.

\medskip
\begin{Remark}\label{remassu} Due to Relaxation Theorem plus Filippov Implicit Function Lemma,
see for instance \cite{AC}, \cite{C},
the integral trajectories of the differential inclusion
\[\dot \zeta \in \ov\co \, g(x, \zeta, A) \qquad\hbox{for  $x$ fixed in $\R^N$,}\]
are locally  uniformly approximated in time by solutions to
\begin{equation}\label{remassu1}
   \dot\eta=g(x,\eta,\al) \qquad\hbox{for some control $\al$.}
\end{equation}
By iteratively applying this property to a concatenation of a sequence of curves of \eqref{remassu1}
for infinitesimal times,
we derive   local bounded time controllability for fast dynamics, namely,
 given $R_1$, $R_2$ positive, there is $T_0=T_0(R_1,R_2)$ such that for
 any   $y_1$, $y_2$  in  $ B(0,R_1)$,  $x \in  B(0,R_2)$,  we can find  a trajectory $\eta$ of
 \eqref{remassu1} joining $y_1$ to $y_2$ in a time $T \leq T_0$.

\end{Remark}

\medskip

\subsection{Controlled dynamics} \label{subdyna}

For any $\eps >0$, any control $\al$, the controlled dynamics  is
defined as

 \begin{equation} \label{dynabis}  \tag{$CD_\eps$}
\left \{
\begin{array}{cc}
           \dot {\xi}(t) =&  \eps \, f(\xi(t),\eta(t),\alpha(t)) \\
           \dot {\eta}(t)  =& g(\xi(t),\eta(t),\alpha(t)) \\
        \end{array} \right .
\end{equation}

\medskip

Notice that if $\xi$, $\eta$ are solutions to \eqref{dynabis} with
initial data  $(x,y)$ then the trajectories
\[ t \mapsto \xi(  t/\eps), \qquad t \mapsto \eta(
t /\eps)\] are solutions to
\begin{equation} \label{dyna} \tag{$\ov{CD}_\eps$}
\left \{
\begin{array}{cc}
           \dot \xi_0(t) =& f(\xi_0(t),\eta_0(t),\alpha(t/\eps)) \\
          \eps \, \dot \eta_0(t)  =& g(\xi_0(t),\eta_0(t),\alpha(t/\eps)) \\
        \end{array} \right .
\end{equation}
with the same initial data.

\smallskip

Given a trajectory $\xi,\eta$ of \eqref{dynabis} with initial data
$(x,y)$ and control $\al$, for some $\eps
>0$, and $T >0$, we deduce from standing assumptions and
Gr\"{o}nwall Lemma, the following basic estimates:
\begin{equation}\label{est1}
    |\xi(t)-x| \leq  Q_0\, T \qquad\hbox{for $t \in [0,T/\eps]$.}
\end{equation}
If $\zeta$ satisfies
\[\dot\zeta = g(x,\zeta,\al) \quad \zeta(0)=y,\]
then
\begin{eqnarray}
  |\eta(T)-\zeta(T)| &\leq&   \int_0^{T} |g(\xi,\eta,\al)
  - g(x,\zeta,\al)| \, \dd s \label{est2}
   \\ &\leq & L_0 \,  \int_0^{T} \big ( |\xi-x| +
   |\eta-\zeta| \big ) \, \dd s
    \leq L_0 \, \eps \, Q_0 \, T^2 \, e^{L_0
   T}. \nonumber
\end{eqnarray}
Finally
\begin{eqnarray}
  |\eta(T)-y| &\leq&   \int_0^{T} |g(\xi,\eta,\al)
  - g(\xi,y,\al)| \, \dd s  + \int_0^{T} |g(\xi,y,\al)| \, \dd s \label{est3}
   \\ &\leq&  L_0 \, R \ T\, e^{L_0
   T}, \nonumber
\end{eqnarray}
where $R$ is an upper bound of $|g|$ in $B(x, \eps \, T) \times
\{y\} \times A$, and similarly
\begin{eqnarray}
  |\eta(T)-y| &\leq&   \int_0^{T} |g(\xi,\eta,\al)
  - g(\xi,\eta(T),\al)| \, \dd s  + \int_0^{T} |g(\xi,\eta(T),\al)| \, \dd s \label{est4}
   \\ &\leq&  L_0 \, R' \ T\, e^{L_0
   T}, \nonumber
\end{eqnarray}
where $R'$ is an upper bound of $|g|$ in $B(x, \eps \, T) \times
\{\eta(T)\} \times A$.

\medskip

By using bounded time controllability condition, we further get:

\smallskip

\begin{Lemma}\label{prestimo} Given   $R_1$,
$R_2$ positive , $x \in B(0,R_1)$, $y$, $z$ in  $B(0,R_2)$, there
is, for any $\eps$, a trajectory $(\xi_\eps,\eta_\eps)$ of
\eqref{dynabis}, starting at $(x,y)$ and a time $T_\eps$ with
\begin{equation}\label{estimo1}
    T_0(R_1,R_2) < T_\eps < 3 \, T_0(R_1,R_2)
\end{equation}
 such that
 \[|\eta_\eps(T_\eps)- z| = \OO (\eps).\]
 The quantity $T_0(\cdot, \cdot)$  is as in Remark \ref{remassu}.
\end{Lemma}
\begin{proof} By controllability condition, see Remark \ref{remassu},
there is a control $\al$ and a
trajectory $\zeta$ with
\begin{equation}\label{estimo2}
    \dot\zeta = g(x,\zeta,\al) \qquad\hbox{for a suitable $\al$}
\end{equation}
starting at $y$ and reaching $z$ in a time $T_\eps \leq
T_0(R_1,R_2)$. Up to adding a cycle based on $z$ and satisfying
\eqref{estimo2} for some control, we can assume $T_\eps$ to satisfy
\eqref{estimo1}. Note that such a cycle does exist again  in force
of the controllability condition. We then take, for any $\eps$, the
trajectories $(\xi_\eps,\eta_\eps)$ of \eqref{dynabis} starting at
$(x,y)$ corresponding to the same control $\al$, and invoke
\eqref{est2} to get the assertion.
\end{proof}

We derive:

\smallskip

\begin{Proposition}\label{straestimo} Given a bounded set  $B$ of  $\R^N \times
\R^M$ and $S >0$, there exists a bounded subset $B_0 \supset B$ such
that for any initial data in $B$ and any $\eps$, we can find
 a trajectory of \eqref{dynabis} lying in $B_0$ as $ t \in
[0,S/ \eps]$.

\end{Proposition}

\begin{proof} We fix $(x,y) \in B$. By \eqref{est1}, we can find
$R_1$, $R_2$  such that  $B \subset B(0,R_1) \times B(0,R_2)$, and
the first component $\xi$ of any trajectory $(\xi,\eta)$ of
\eqref{dynabis}, for any $\eps$, starting at $(x,y)$ is contained in
$B(0,R_1)$. We write $T_0$ for $T_0(R_1,R_2)$. Clearly, it is enough
to establish the assertion for $\eps$ small.

By applying Lemma \ref{prestimo} with $\eps$ suitably small and
$z=0$, we find a time $T_\eps$ and a trajectory
$(\xi_\eps,\eta_\eps)$ of \eqref{dynabis} such that
$(\xi_\eps(T_\eps),\eta_\eps(T_\eps)) \in B(0,R_1) \times B(0,R_2)$.
Taking into account  that the time $T_\eps$ is estimated from above
and below by a positive quantity, see \eqref{estimo1}, we can
iterate the procedure and get by concatenation of the curves so
obtained, a trajectory $(\xi_0,\eta_0)$  in $[0,t_0/\eps]$, starting
at $(x,y)$, with the crucial property that there are times
$\{t_i\}$, $i=1, \cdots k$, for some index $k$, in $[0,S/\eps]$ such
that
\begin{enumerate}
    \item[] for any $t \in[0,S/\eps]$, there is $t_i$ with $|t- t_i| \leq 3 \, T_0$;
    \item[] $\eta_\eps(t_i) \in B(0,R_2)$ for any $i$.
\end{enumerate}
We derive as $t \in \left [ 0, \frac S\eps \right ]$
\begin{eqnarray}
  |\xi_\eps(t) - x_0| &< & Q_0 \, S \label{estimo5} \\
  |\eta_\eps(t)| &\leq& R_2 + 3 \, P \, T_0
   \label{estimo4}
\end{eqnarray}
with  constant $P$ solely depending, see \eqref{est3},  upon $R_1$,
$R_2$, $T_0(R_1,R_2)$. This proves the assertion.

\end{proof}

\medskip

The next result is a strengthened version of  Lemma \ref{prestimo}
stating that the approximation of a  value of the fast variable by a
trajectory of the fast dynamics can be realized in any predetermined
suitably large time. To establish it, we need exploiting total
controllability assumption {\bf (H3)} in its full extent.  The lemma
will be used in the proof of Theorem \ref{maindue}.

\smallskip

\begin{Lemma}\label{prestimobis}
Given   $x \in \R^N$, $y$, $z$ in $\R^M$, and $S >0$ suitably large,
there is, for any $\eps$, a trajectory $(\ov\xi_\eps,\ov\eta_\eps)$
of \eqref{dynabis}, starting at $(x,y)$
 such that
 \[|\ov\eta_\eps(S)- z| = \OO (\eps).\]
\end{Lemma}
\begin{proof}
We fix $R_1$, $R_2$ such that $x \in B(0,R_1)$, and $y$, $z$ are in
$B(0,R_2)$. We take $S$ with $S > 3 \, T_0(R_1,R_2)$. By applying
Lemma \ref{prestimo}, we find $T_\eps < 3 \, T_0(R_1,R_2) < S$ and,
for any $\eps$,  a curve $(\xi_\eps,\eta_\eps)$ of \eqref{dynabis}
starting at $(x,y)$ with
\[|\eta_\eps(T_\eps) - z| = \OO( \eps).\]
By  iterating the procedure, if necessary, as in the proof of  Lemma
\ref{prestimo}, we can extend it to an interval $[0,S_\eps]$, with
$S - S_\eps <T_\eps$, still getting
\begin{equation}\label{prestimobis1}
    |\eta_\eps(S_\eps) - z| = \OO( \eps).
\end{equation}
By {\bf (H3)} and Relaxation Theorem, see Remark \ref{remassu}, we
find a control $\be$ and a trajectory $\zeta_\eps$ satisfying
\[ \dot\zeta_\eps=g(\xi_\eps(S_\eps),\zeta_\eps,\be) \quad \zeta_\eps(0)=\eta_\eps(S_\eps)\]
with
\begin{equation}\label{prestimobis2}
    |\zeta_\eps(t) - \eta_\eps(S_\eps)| = \OO(\eps) \qquad\hbox{for $t
\in [0,S-S_\eps]$.}
\end{equation}
 Owing to \eqref{est2}, the trajectory
$(\xi^0_\eps,\eta^0_\eps)$ of \eqref{dynabis} starting at
$(\xi_\eps(S_\eps), \eta_\eps(S_\eps))$, with control $\be$
satisfies
\begin{equation}\label{prestimobis3}
    |\eta_\eps^0(S-S_\eps) - \zeta_\eps(S-S_\eps)| = \OO(\eps).
\end{equation}
By concatenation of $\eta_\eps$ and $\eta^0_\eps$, we finally get,
in force of \eqref{prestimobis1}, \eqref{prestimobis2},
\eqref{prestimobis3}, a trajectory satisfying the assertion.
\end{proof}

\medskip

\subsection{Minimization problems and value functions}
\label{subvalue} \; We consider for any $(x,y) \in \R^N \times
\R^M$, $t
>0$, $\eps >0$, the optimization problems

 \begin{equation} \label{minprob}
      \inf_\alpha \eps \,  \int_0^{\frac t \eps}\ell \big (\xi_\eps,
      \eta_\eps, \alpha \big )\, \dd s +
      u_0  \left (\xi_\eps \left (\frac t\eps \right ),
      \eta_\eps \left (\frac t\eps \right ) \right )
\end{equation}
with $\xi_\eps$, $\eta_\eps$ are solutions to \eqref{dynabis} in
$[0,+ \infty)$,  issued from the  initial datum $(x,y)$. Or
equivalently with the change of variables $r = \eps \, s$
\begin{equation} \label{minprobbis}
      \inf_\alpha \; \int_0^{t} \ell \big (\xi^0_\eps,
      \eta^0_\eps, \alpha \big )\, \dd r +
      u_0  (\xi^0_\eps(t), \eta^0_\eps(t) ) )
\end{equation}
with $\xi^0_\eps$, $\eta^0_\eps$ are solutions to \eqref{dyna} in
$[0,+ \infty)$,  issued from $(x,y)$. We denote by $V^\eps$ the
corresponding value functions, namely the functions associating to
any initial datum $(x,y)$ and time $t$ the infimum of the functional
in \eqref{minprob}/ \eqref{minprobbis}. They are apparently
continuous with respect to all arguments.

\smallskip

\begin{Remark} Looking at the form of the above minimization
problem, we  understand that coercivity assumption {\bf (H4)} plus
{\bf (H5)} plays the role of a compactness condition for the fast
variable,  inasmuch as  it implies that the trajectories of the fast
dynamics realizing the value function, up to some small constant,
lie in a compact subset of $\R^M$. This fact will be crucial in the
asymptotic analysis.

\end{Remark}

\medskip

We derive from Proposition \ref{straestimo}:

\smallskip

\begin{Proposition}\label{estimo} The value functions $V^\eps$
are locally equibounded.
\end{Proposition}
\begin{proof}

Let $C$ be a bounded set of $\R^N \times \R^M \times [0, + \infty)$,
and $(x_0,y_0,t_0) \in C$. Thanks to Proposition \ref{straestimo},
there are for any $\eps$ trajectories $(\xi_0, \eta_0)$, we drop the
dependence on $\eps$ to ease notations,  of \eqref{dynabis} starting
at $(x_0,y_0)$, and contained in a bounded set of $\R^N \times \R^M$
solely depending on $C$.
 By using the formulation
\eqref{minprob} of the minimization problem, we get
\[
    V^\eps(x_0,y_0,t_0) \leq \eps \, \int_0^{\frac {t_0}\eps} \ell( \xi_0(s),
\eta_0(s), \al(s)) \, \dd s + u_0(\xi_0(t_0/\eps),
\eta_0(t_0/\eps)).
\]
 Since the integrand   in the above formula and $u_0$  are
bounded independently of $\eps$, we obtain the equiboundedness from
above of the $V^\eps$.

We now consider any trajectory $(\xi,\eta)$ of \eqref{dynabis}
starting  $(x_0,y_0)$ and corresponding to a control $\beta$. By
\eqref{est1}, $\xi(t)$ lies in a compact subset $K$ of $\R^N$, only
depending on $C$, for $t \in [0,t_0/\eps]$, and by coercivity
assumption {\bf (H4)}, there is a constant $P_0$ with
\begin{equation}\label{estimo7}
   \ell(x,y,a) \geq P_0 \qquad\hbox{for any $(x,y,a) \in K \times \R^M
\times A$.}
\end{equation}
Since $- Q_0$ is a lower bound of $u_0$ in $\R^N \times \R^M$, see
{\bf (H5)}, this implies
\begin{eqnarray}
   && \eps \, \int_0^{\frac {t_0}\eps} \ell( \xi(s),
\eta(s), \beta(s)) \, \dd s + u_0(\xi(t_0/\eps), \eta(t_0/\eps)) \geq \label{estimo8}\\
 && \eps \, \frac{t_0}\eps \,P_0 + u_0(\xi(t_0/\eps), \eta(t_0/\eps)) \geq P_0 \,t_0 - Q_0 .
 \nonumber
\end{eqnarray}
Being $(\xi,\eta)$ an arbitrary  trajectory with initial point
$(x_0,y_0)$, the above inequality shows the claimed local
equiboundedness from below of value functions.

\end{proof}

\medskip

The previous result  allows us to define ${\limsup}^\#V^\eps$,
${\liminf}_\#V^\eps$, these functions will be denoted by $\ov V$,
$\underline V$, respectively, in what follows. The next proposition
shows that they only depend on time and slow variable, at least for
positive times.

\smallskip

\begin{Proposition}\label{perez} We have
\begin{eqnarray*}
  ({\liminf}_\#V^\eps)(x_0,y_0,t_0)
 &=&({\liminf}_\#V^\eps)(x_0,z_0,t_0)=:\underline V(x_0,t_0) \\
  ({\limsup}_\#V^\eps)(x_0,y_0,t_0)
 &=&({\limsup}_\#V^\eps)(x_0,z_0,t_0)=: \ov V(x_0,t_0)
\end{eqnarray*}
for any $x_0 \in \R^N$, $y_0$, $z_0$ in $\R^M$ and $t_0 >0$.
\end{Proposition}
\begin{proof} We start by
\smallskip

\noindent {\bf Claim:} {\em  Given  positive constants $R_1$, $R_2$,
$S$ we can determine $P=P(R_1,R_2,S)
>0$ such that for any  $\eps >0$, $x \in B(0,R_1)$, $y$, $z$ in $B(0,R_2)$, $t
\in [0,S]$ there exist $x'$, $x''$, $z'$, $z''$, $t'$, $t''$,
depending on $\eps$, with
\[|x-x'| <  \eps \, P, \quad |z-z'| < \eps \, P, \quad|t -t'| < \eps \, P,\]
\[|x-x''| < \eps \, P, \quad |z-z''| <\eps \, P, \quad|t -t''| < \eps \, P\]
such that
\begin{eqnarray*}
  V^\eps(x',z',t') &<& V^\eps(x,y,t) + \eps \, P \\
  V^\eps(x'',z'',t'') &>& V^\eps(x,y,t) - \eps \, P.
\end{eqnarray*}}

\smallskip

We fix $\eps$. By controllability assumption (see Remark
\ref{remassu}) $z$ and $y$ can be joined in a time $T$ less than or
equal to $T_0 =T_0(R_1,R_2)$ by a curve $\zeta$ satisfying
\[\dot \zeta=g(x,\zeta,\al) \qquad\hbox{for a suitable control
$\al$.}\] We consider the trajectory $(\xi, \eta)$ of
\eqref{dynabis} with the same control $\al$ satisfying
\[ \xi(T) = x \quad\hbox{and} \quad  \eta(T)= y,\]
and set
\[x'= \xi(0) \quad\hbox{and} \quad  z'= \eta(0).\]
 By \eqref{est1}, \eqref{est2}, we get
 \begin{eqnarray}
   |x'- x|  &<& \eps \, P_0  \label{perez0}\\
   |z'- z| &<&  \eps \, P_0 \label{perez00}
 \end{eqnarray}
for a suitable $P_0 >0$.  We select  a trajectory $(\xi_0,\eta_0)$
of \eqref{dynabis} with initial datum $(x,y)$,   corresponding to a
control $\be$, such that
\begin{equation}\label{perez1}
    V^\eps(x,y,t) \geq \eps \, \int_0^{\frac {t}\eps} \ell( \xi_0,
\eta_0, \be) \, \dd s + u_0\left(\xi_0 \left (\frac t\eps \right ),
\eta_0 \left (\frac t\eps \right ) \right ) - \eps.
\end{equation}
We set
\begin{equation}\label{perez11}
    t'= t + \eps \, T,
\end{equation}
by concatenation of $\al$ and $\be$, $\xi$ and $\xi_0$, $\eta$ and
$\eta_0$,  we get a control $\ga$ and trajectory $(\ov\xi, \ov\eta)$
of \eqref{dynabis} starting at $(x',z')$, defined in $\left [0,
\frac{t'}\eps \right ]$.  We consequently have
\begin{eqnarray*}
&& V^\eps(x',z',t') \leq  \\& &\eps \, \int_0^{\frac {t'}\eps}
\ell(\ov \xi, \ov\eta, \ga) \, \dd s  + u_0 \left( \ov\xi \left (
\frac {t'}\eps  \right ), \ov\eta
\left ( \frac {t'}\eps  \right ) \right ) =  \\
  & &\eps \, \int_0^{T} \ell( \xi,
\eta, \al) \, \dd s  + \eps \, \int_{T}^{\frac {t '}\eps}
\ell(\xi_0(s-T), \eta_0(s-T), \be(s-T)) \, \dd s  + \\ & & u_0
\left( \xi_0 \left ( \frac {t'}\eps  \right ), \eta_0 \left ( \frac
{t'}\eps  \right ) \right ).
\end{eqnarray*}
By taking into account \eqref{est2} and \eqref{perez1},  we derive
\begin{equation}\label{perez55}
    V^\eps(x',z',t') \leq \eps \, Q \, T_0 + V^\eps(x,y,t) + \eps \quad\hbox{for a suitable $Q >0$.}
\end{equation}
 The first part of the
claim is therefore proved taking into account \eqref{perez0},
\eqref{perez00}, \eqref{perez11}, \eqref{perez55}, and defining
\[P = \max \{P_0,T_0, Q \, T_0+1 \}.\]
\smallskip
The estimates for $x''$, $y''$, $z''$, $t''$ can be obtained
slightly modifying   the above argument. We sketch the proof for
reader's convenience. We denote by $\zeta'$ a curve joining $y$ to
$z$ in a time $T' \leq T_0$ and satisfying
\[\dot{\zeta'}=g(x,\zeta',\alì) \qquad\hbox{for a suitable control
$\al'$.}\] We  consider the trajectory $(\xi', \eta')$ of
\eqref{dynabis} with the same control $\al'$ satisfying
\[ \xi'(0) = x \quad\hbox{and} \quad  \eta'(0)= y,\]
and set
\[x''= \xi'(T') \quad\hbox{and} \quad  z''= \eta'(T').\]
 As in the first part of the proof we get
 \begin{eqnarray*}
   |x''- x| &\leq& P_0 \, \eps \\
   |z''- z| &\leq& P_0 \, \eps,
 \end{eqnarray*}
for a suitable $P_0$. We select  a trajectory $(\xi_0',\eta_0')$  of
\eqref{dynabis} with initial datum $(x'',z'')$,   corresponding to a
control $\be'$, which is optimal for $V^\eps(x'',z'',t - \eps \,
T')$ up to $\eps$, namely
\[ V^\eps(x'',z'',t '') \geq \eps \, \int_0^{\frac {t''}\eps} \ell( \xi'_0,
\eta'_0, \be') \, \dd s + u_0\left(\xi_0' \left (\frac {t''}\eps
\right ) ,\eta_0' \left (\frac {t''}\eps \right ) \right ) - \eps.
\] Here we are assuming $\eps$ so small that $t'':=t - \eps \, T'$
is positive, this does not entail  any limitation to the argument
since we are interested to $\eps$ infinitesimal.
 From this point we go on as in the previous part.
\smallskip

We exploit the first part of the claim to show that for any pair of
values $y_0$, $z_0$ of the fast variable, any $x_0 \in \R^N$, $t_0
>0$
\begin{equation}\label{perez2}
    ({\liminf}_\# V^\eps)(x_0,z_0,t_0) \leq ({\liminf}_\# V^\eps)(x_0,y_0,t_0),
\end{equation}
which in turn implies  by the arbitrariness of $y_0$, $z_0$, that
${\liminf}_\# V^\eps$ independent of the fast variable. We consider
$\eps_n$, $x_n$, $y_n$, $t_n$ converging to $0$, $x_0$, $y_0$,
$t_0$, respectively, with
\[\lim_n V^{\eps_n}(x_n,y_n,t_n)= ({\liminf}_\#
V^\eps)(x_0,z_0,t_0). \] Since all the $x_n$ , $y_n$, and  $z_0$,
$t_n$ are contained   in compact subsets  of $\R^N$, $\R^M$, $[0,+
\infty)$, respectively, we can apply, for any given $n \in \N$,
 the claim to $\eps= \eps_n$, $x=x_n$, $y=y_n$, $z=z_0$, $t=t_n$ and get of $x'_n$,
 $z'_n$, $t'_n$ with
 \[|x_n-x'_n| < \eps_n \, P, \quad |z_0-z'_n| <  \eps_n \, P, \quad|t_n -t'_n| <
\eps_n \, P\] and
\[V^{\eps_n}(x'_n,z'_n,t'_n) < V^{\eps_n}(x_n,y_n,t_n) + \eps_n \, P \]
for a suitable $P$. Sending $n$ to infinity we deduce
\[ \liminf V^{\eps_n}(x'_n,z'_n,t'_n) \leq \lim V^{\eps_n}(x_n,y_n,t_n)
= ({\liminf}_\# V^\eps)(x_0,z_0,t_0),\] which implies \eqref{perez2}
since $x'_n \to x_0$, $z'_n \to z_0$ and $t'_n \to t_0$.

The assertion relative to ${\limsup}^\# V^\eps$ is obtained using
the second part of the claim and slightly adapting the above
argument.

\end{proof}

\medskip

As a consequence of coercivity of running cost assumed in {\bf (H4)}
we deduce:

\smallskip

\begin{Proposition}\label{estimobis} The value function  $V^\eps$
satisfy for any $\eps$, any compact  subset $K$ of  $\R^N \times (0,
+ \infty)$
\[ \lim_{|y| \to + \infty} \min_{(x,t) \in K} V^\eps(x,y,t) = + \infty.\]
\end{Proposition}
\begin{proof} We fix $\eps$, we assume, without loosing any generality, that $K$ is of the form
$ \widetilde K \times [S,T]$, where $\widetilde K$ is a compact
subset of $\R^N$ and $S$, $T$ are positive times. Given any $P >0$,
we can determine by {\bf (H4)} a constant $R$ such that the  ball
$B(0,R)$ of $\R^M$ satisfies
\begin{equation}\label{coercizzo1}
    \ell(x,y,a) > P \qquad\hbox{for any $(x,a) \in \widetilde K \times A$, $y
\in \R^M \setminus B(0,R)$.}
\end{equation}
 Taking into account the estimate
\eqref{est4}, we see that  there exists $R_0 > R$ such that
\begin{equation}\label{coercizzo2}
   \eta(t) \not\in B(0,R) \qquad\hbox{for $t \in [0,T]$}
\end{equation}
for any trajectory of \eqref{dynabis} starting in $K_0 \times \big
(\R^M \setminus B(0,R_0) \big )$. Given $\de >0$, we find, for any
\[(x,y,t) \in K_0 \times \big (\R^M \setminus B(0,R_0) \big ) \times [S,T]\]
 a trajectory $(\xi_0,\eta_0)$ of \eqref{dynabis}, corresponding
 to a control $\al$, starting at $(x,y)$
  with
\[V^\eps(x,y,t) \geq \eps \, \int_0^{\frac {t}\eps} \ell( \xi_0,
\eta_0, \al) \, \dd s + u_0\left(\xi_0 \left (\frac t\eps \right ),
\eta_0 \left (\frac t\eps \right ) \right ) - \de. \] We deduce by
\eqref{coercizzo1}, \eqref{coercizzo2}, {\bf (H5)}
\[V^\eps(x,y,t) \geq P \, S - Q_0 - \de,\]
which gives the assertion, since $P$ can be chosen as large as
desired, and $\de$ as small as desired.
\end{proof}

\medskip

 \subsection{ HJB equations} We define the Hamiltonian

\[ H(x,y,p,q)= \max_{a \in A} \{- p \cdot f(x,y,a) - q \cdot
g(x,y,a) - \ell(x,y,a)\}\]

\smallskip

The main contribution of Assumption {\bf (H3)} is the following
coercivity property on $H$:

\smallskip

\begin{Lemma}\label{astima} For any given  bounded set $C \subset \R^N
\times \R^M \times \R^N$,  we have
\[ \lim_{|q| \to + \infty} \,  \min_{(x,y,p) \in C} H(x,y,p,q) = + \infty .\]
\end{Lemma}
\begin{proof}  We denote by $r$ the positive constant provided by  {\bf (H3)}
in correspondence to the projection of $C$ on the  state variables
space $\R^N \times \R^M$. We consequently have for $(x,y)$ in such
projection and $q \in \R^M$
\begin{equation}\label{astima1}
 \max \{ q \cdot v \mid v \in g(x,y,A)\} = \max \{ q \cdot v \mid v \in \ov \co \, g(x,y,A)\}
\geq r \, |q|.
\end{equation}
We take  $(x,y,p) \in C$, and denote by $a_0$ an element in the
control set such that $g(x,y,a_0)$ realizes the maximum in
\eqref{astima1}. We get from the very definition of $H$ and
\eqref{astima1}
\[ H(x,y,p,q) \geq - |p| \, |f(x,y,a_0)| + r \, |q| - |\ell(x,y,a)|
\qquad\hbox{for any $q$}.\] When  we send $|q|$  to infinity, all
the terms in the right hand--side of the above formula stay bounded
except $r \, |q|$. This gives the assertion.

\end{proof}

Given a bounded set $B$ in $\R^N \times \R^M$, one can check by
direct calculation that  $H$ satisfies
\begin{eqnarray}
 && |H(x_1,y_1,p,q) -H(x_2,y_2,p,q)| \leq  \label{compara} \\
 && L_0 \, (|x_1-x_2| + |y_1-y_2|) (|p| + |q|) +
 \om(|x_1-x_2|+ |y_1-y_2|) \nonumber
\end{eqnarray}
 for any $(x_1,y_1)$, $(x_2,y_2)$ in $B$ and $(p,q) \in \R^N \times \R^M$, where
 $\om$ is an uniform continuity modulus of $\ell$ in $B \times A$ and $L_0$ is as in {\bf (H2)}.
 We also have
\begin{eqnarray}
 && |H(x,y,p_1,q_1) -H(x,y,p_2,q_2)| \leq  \label{comparabis} \\
 && |f(x,y,a_0)| \, |p_1-p_2| + |g(x,y,a_0)| \, |q_1-q_2| \nonumber
\end{eqnarray}
for any $(x,y) \in \R^N \times \R^M$, $(p_1,q_1)$, $(p_2,q_2)$ in
$\R^N \times \R^M$, a suitable $a_0 \in A$.

\medskip

We write, for any $\eps >0$,   the family  of
Hamilton--Jacobi--Bellman problems
\begin{equation}\label{Hje} \tag{HJ$_\eps$}
    \left \{ \begin{array}{cc}
  u^\eps_t + H \left (x,  y , D_x u^\eps, \frac{D_y u^\eps}\eps \right ) &= \; 0 \\
   u^\eps(x,y,0) & \qquad = \; u_0(x,y) \\
\end{array} \right .
\end{equation}

\smallskip
 It is well known that    the value
functions $V^\eps$ are solutions to \eqref{Hje}, even if not
necessarily unique in our setting. However, due to the estimate
\eqref{compara},  we have the following local comparison result (see
for instance \cite{B}):

\smallskip

\begin{Proposition}\label{comparacompara} Given a bounded open set
$B$ of $\R^N \times \R^M$ and times $t_2 >t_1$, let $u$, $v$ be
continuous subsolution and supersolution, respectively, of the
equation in  \eqref{Hje}. If $u \leq v$ in $ \partial_p \big (B
\times (t_1,t_2) \big )$ then $u \leq v$ in $B \times (t_1,t_2)$,
where $\partial_p$ stands
 for the parabolic boundary.

\end{Proposition}

\medskip

We define  the {\em effective Hamiltonian}
\begin{equation}\label{effe}
   \overline H(x,p)= \inf \{ b  \in \R \mid H(x,y,p,Du)=
b \;\hbox{admits a  subsolution in $\R^M$}\}
\end{equation}
for any fixed $(x,p) \in \R^N \times \R^N$, where the equation
appearing in the formula  is solely in the fast variable $y$ with
slow variable $x$ and corresponding momentum $p$ frozen. This
quantity can be in principle infinite, however we will show in what
follows that not only  it is finite for any $(x,p)$, but also that
the infimum is actually a minimum.

\medskip

We write the limit equation

  \begin{equation}\label{Hjlim} \tag{$\overline{\rm HJ}$}
u_t + \overline H  (x, D u ) = 0.
\end{equation}

\bigskip

\section{Cell problems}\label{cell}

\parskip +3pt

The section is devoted to the analysis of the stationary
Hamilton--Jacobi equations in $\R^M$ appearing in the definition of
effective Hamiltonian, namely with slow variable and corresponding
momentum frozen.

\subsection{Basic analysis} \label{basic} \; We   fix  $(x_0, p_0) \in \R^N \times \R^N$, and set to ease notations

\begin{eqnarray*}
  H_0(y,q) &=& H(x_0,y,p_0,q) \qquad\qquad\qquad\qquad\hbox{for any $(y,q) \in \R^M \times
\R^M$} \\
  \ell_0(y,a) &=& \ell(x_0,y, a) + p_0 \cdot f(x_0,y,a) \,\;\qquad\hbox{for any $(y,a)
  \in \R^M \times A$} \\
  g_0(y,a) &=& g(x_0,y,a) \qquad\qquad\qquad\qquad\qquad\hbox{for any $(y,a)
  \in \R^M \times A$}
\end{eqnarray*}
Given a control $\al(t)$, we consider the controlled differential
equation in $\R^M$
\begin{equation}\label{cedo}
   \deta(t)= g_0(\eta(t),\al(t)).
\end{equation}

\medskip

We directly derive from Lemma \ref{astima}:

\smallskip

\begin{Lemma}\label{coerci} We have
\[\lim_{|q| \to + \infty} \, \min_{y \in K} H_0(y,q)= + \infty\]
for any compact subset $K$ of $\R^M$.
\end{Lemma}

\smallskip

This result implies, according to Lemma \ref{equili}, that all
subsolutions are locally Lipschitz--continuous, and allows adopting
the metric method, see Appendix \ref{appendice}, in the analysis of
the cell equations. To ease notation, we set $c_0= \ov H(x_0,p_0)$,
also called the critical value of $H_0$, see \eqref{cricri}. We will
prove in Proposition \ref{critic} that $c_0$ is finite. We denote by
$Z$, $\si$, $S$ the corresponding sublevels, support function and
intrinsic distance, see Appendix \ref{appendice} for the
corresponding definitions. Same objects for a supercritical value
$b$ will be denoted by $Z_b$, $\si_b$, $S_b$.

\medskip

To compare the metric and control--theoretic viewpoint, we notice
\[Z_b(y)= \{ q \in \R^M \mid q \cdot (-g_0(y,a)) \leq \ell_0(y,a)+ b \;\hbox{for any $a \in A$}\}.
\]
for any given supercritical $b \in \R$, namely $b \geq c_0$,  and $y
\in \R^M$. This implies that the support function $\si_b(y,\cdot)$
is the maximal subadditive positively homogeneous function $\rho:
\R^M \to \R$ with
\begin{equation}\label{support2}
    \rho(- g_0(y,a)) \leq \ell_0(y,a) + b  \qquad\hbox{for any $a \in A$,}
\end{equation}
which somehow justifies the  next equivalences.
\smallskip
\begin{Proposition}\label{criticone} Given a supercritical value $b$, the following conditions are equivalent:
\begin{itemize}
    \item [{\bf (i)}] $u$ is a subsolution to $H_0=b$;
    \item [{\bf (ii)}] $u(y_2)-u(y_1) \leq S_b(y_1,y_2)$ for any $y_1$, $y_2$;
    \item [{\bf (iii)}] $u(y_1)-u(y_2) \leq \int_{0}^{T}
(\ell_0(\eta(t),\al(t)) + b ) \, \dd t$ for any $y_1$, $y_2$,
 time $T$, control $\al$, any trajectory $\eta$ of
\eqref{cedo} with $\eta(0)=y_1$, $\eta(T)=y_2$.
\end{itemize}
\end{Proposition}
\begin{proof} The equivalence {\bf (i)} $\Longleftrightarrow$ {\bf (ii)} is given in Proposition
\ref{subsubsub} {\bf (i)},
  the equivalence {\bf (i)} $\Longleftrightarrow$ {\bf (iii)} is the usual
  characterization of subsolutions to Hamilton--Jacobi--Bellman equations in terms of suboptimality,
  see \cite{BCD}.
\end{proof}

\medskip
One advantage of metric method is that any curve is endowed of a
length, while integral cost functional is only defined on
trajectories of the controlled dynamics. Also notice that there is a
change of orientation between length and cost functional, that can
detected from \eqref{support2} and comparison between items {\bf
(ii)} and {\bf (iii)} in Proposition \ref{criticone}. This just
depends on $u_0$ being terminal cost and initial condition in
\eqref{Hje}, the discrepancy should be eliminated if \eqref{Hje}
were posed in $(- \infty, 0)$ and $u_0$ should consequently play the
role of terminal condition and initial cost.

\smallskip

\begin{Proposition}\label{critic} The critical value $c_0$ is
finite.
\end{Proposition}

\begin{proof}  Owing to coercivity of $\ell$
 and boundedness of  $f$
\[H_0(y,0)= \max_{a \in A} \{ -
\ell_0(y,a)\} \to - \infty \quad\hbox{as $|y| \to + \infty$}\] and
consequently
\[ H_0(y,0) < 0  \quad\hbox{outside some compact subset $
K$ of $\R^M$.}\]  We set
\begin{eqnarray*}
  b_0 &=& \max \big \{0, \max \{H_0(x,0) \mid x \in K\} \big \},
\end{eqnarray*}
then the null function is subsolution to $H_0=b_0$ in $\R^M$,  and
so $c_0 < + \infty$.

By controllability condition {\bf (H3)}, we find a cycle $\eta$
defined in $[0,T]$, for a positive $T$, solution to \eqref{cedo} for
some control $\al$.  We put
\[R = \int_0^{T} \ell_0(\eta, \al)  \, \dd t, \]
and for $b < - \frac R {T}$ we get
\[ \int_0^{T}  (\ell_0(\eta, \al)  + b ) \; \dd t < R - \frac R {T} \, T = 0. \]
The above cycle, repeated infinite times, gives a trajectory of
\eqref{cedo} in $[0, + \infty)$, still denoted by $\eta$,    such
that
\begin{equation}\label{critic1}
    \int_0^{\infty}  (\ell_0(\eta, \al)  + b) \; \dd t = - \infty.
\end{equation}
If there were a subsolution $u$ to $H_0=b$ then
\begin{equation}\label{critic2}
   u(\eta(0)) - u(\eta(t_0)) \leq \int_{0}^{t_0}
(\ell_0(\eta(t),\al(t)) + b ) \, \dd t \quad\hbox{for any $t_0>0$.}
\end{equation}
 But the support
of $\eta$ is  equal to $\eta([0,T])$ which is a compact subset of
$\R^M$, so that the oscillation of $u$ (which  is locally Lipschitz
continuous) on it is bounded. This shows that \eqref{critic1} and
\eqref{critic2}  are in contradiction. We then deduce  that the
equation $H_0= b$ cannot have any subsolution, showing in the end
that  $c_0 > - \infty$.
\end{proof}

\medskip

 We deduce from standing assumptions a sign and a coercivity condition  on the critical
 distances. To do that, we start selecting a compact set $C$  of $\R^M$  with
\begin{equation} \label{dolo1}
 H_0(y,0) = - \min_{a \in A} \ell_0(y,a) < c_0 - Q_0  \qquad\hbox{for  any $y \in \R^M \setminus
  C$,}
\end{equation}
where $Q_0$ is as in {\bf (H2)}. This is possible since
$\ell_0$ is coercive. Further we set
\begin{equation}\label{dolo40}
    K_0=  \left \{y \mid d(y,C) \leq \max_{C \times C} \, |S| \right \}.
\end{equation}

\smallskip

\begin{Proposition}\label{cortre} The following properties hold true:
\begin{itemize}
    \item[{\bf (i)}] $ \lim_{|y| \to + \infty} \inf_{y_0 \in K}S(y_0,y) = +
    \infty $\; for any compact set $K \subset \R^M$;
\item [{\bf (ii)}] $Z(y) \supset B(0,1)$ for any $y$ outside the
compact set  $K_0$ defined as in \eqref{dolo40};
\item [{\bf (iii)}] $S(y_1,y_2) >0$  \; for   any pair $y_1$, $y_2$
    outside  $K_0$.

\end{itemize}

\end{Proposition}

\begin{proof}  If $q \in \R^M$ satisfies
 \begin{equation}\label{dolo10}
   H_0(y,q)=c_0 \qquad\hbox{ for some $y$ in $\R^M \setminus C$,}
\end{equation}
where $C$ is defined as in \eqref{dolo1},  then
\[ c_0 = H_0(y,q) = \max_{a \in A} \{ - g_0(y,a) \cdot q - \ell_0(y,a)
\} \leq Q_0 \, |q| -  \min_{a \in A} \ell_0(y,a) \]
 and   by the very definition of $C$
\begin{equation}\label{dolo11}
    |q| \geq   \frac{c_0 +  \min_{a \in A} \ell_0(y,a)}{Q_0} > \frac{Q_0}{Q_0} = 1,
\end{equation}
Since $0$ in the interior of $Z(y)$ by \eqref{dolo1}, we derive a
stronger version of item {\bf (ii)}, with $C$ in place of $K_0$,
 which in turn implies
\[  \frac{v}{|v|} \in Z(y) \qquad \hbox{for any $y \in \R^M
    \setminus C$, $v \in \R^M$ with $v \neq 0$}\]
and consequently
\begin{equation}\label{dolo3}
   \si(y,v) \geq v \cdot \left ( \frac{v}{|v|} \right ) =  |v| \qquad \hbox{for any $y \in \R^M
    \setminus C$, $v \in \R^M$ with $v \neq 0$.}\
\end{equation}
Next, we  fix  a compact set $K$ and consider two points $y_1 \in
K$, $y_2 \not\in C$ and any curve $\zeta$, defined in $[0,1]$,
linking them. We distinguish two cases according on whether the
intersection of $\zeta$ with  $C$ is nonempty or empty. In the first
instance we set
\begin{eqnarray}
 \label{dolo33} t_1 &=& \min \{t \in [0,1] \mid \zeta(t) \in C\} \\
  t_2 &=& \max \{t \in [0,1] \mid \zeta(t) \in C\}. \label{dolo34}
\end{eqnarray}
We denote by $R$ an upper bound of $|S|$ in $C \times C$ and
exploit  \eqref{dolo3} to get
\begin{eqnarray}
 \label{dolo4} \int_{0}^{1} \si(\zeta, \dot \zeta) \, \dd t &=&
  \int_0^{t_1} \si(\zeta, \dot \zeta)
  \, \dd t
  +\int_{t_1}^{t_2} \si(\zeta, \dot \zeta) \, \dd t+  \int_{t_2}^{1} \si(\zeta, \dot \zeta) \,
  \dd t\\
\nonumber &\geq&  |y_1 - \zeta(t_1)| + S(\zeta(t_1), \zeta(t_2)) +
 |y_2 - \zeta(t_2)| \\ \nonumber &\geq& -R + d(y_1,C) + d(y_2,C) .
\end{eqnarray}
If instead the curve $\zeta$ entirely lies outside $C$, we have by
\eqref{dolo3}
\begin{equation}\label{dolo5}
    \int_{0}^{1} \si(\zeta, \dot \zeta) \, \dd t \geq
|y_1-y_2|.
\end{equation}
In both cases we get item {\bf (i)} sending $y_2$ to infinity and
taking into account that $y_1$ has been arbitrarily chosen in $K$.

We finally see, looking at  \eqref{dolo4}, \eqref{dolo5}, and
slightly adapting the above argument that $K_0$, defined as in
\eqref{dolo40}, satisfies item {\bf (iii)}.
\end{proof}

\medskip

\begin{Remark} \label{remcortre} Given a  compact set $K \subset \R^M$,
the same argument of Proposition \ref{cortre} allows also proving
\begin{equation}\label{dolo6}
     \lim_{|y| \to + \infty} \inf_{y_0 \in K} S(y,y_0) = + \infty
\end{equation}
\end{Remark}

\medskip

\begin{Corollary}\label{corquattro}
For any  bounded open set $B$ there exists $R >0 $ such that if
$y_1$, $y_2$ belong to $B$ then all $1$--optimal curves for
$S(y_1,y_2)$ are contained in $B(0,R)$.
\end{Corollary}
\begin{proof} We can assume without loosing generality that $B \supset K_0$, where
$K_0$ is the set defined in \eqref{dolo40}. We set
\[ P = \sup_{B \times B} |S|.\]
 By Proposition \ref{cortre} {\bf (i)} there is $R$ such that
\[ \inf_{y_0 \in B} S(y_0,y)  > 2 \, P   +2 \qquad\hbox{for  $y$ with $|y| > R$.}\]
We claim that such an $R$ satisfies the claim. In fact, assume by
contradiction that there are $y_1$, $y_2$ in $B$ and an $1$--optimal
curve $\zeta$, defined in $[0,1]$, for $S(y_1,y_2)$ not contained in
$B(0,R)$. Let $t_1$ be a time in $(0,1)$ with $\zeta(t_1) \not \in
B(0,R)$ and set
\[t_2 = \min\{ t \in (t_1,1) \mid \zeta(t) \in K_0 \subset B\}\]
 then, taking into account Proposition \ref{cortre}
\begin{eqnarray*}
 S(y_1,y_2) &\geq&  \int_{0}^{1} \si(\zeta, \dot \zeta) \, \dd t -1
\\ &=& \int_0^{t_1} \si(\zeta, \dot \zeta)
  \, \dd t
  +\int_{t_1}^{t_2} \si(\zeta, \dot \zeta) \, \dd t+  \int_{t_2}^{1} \si(\zeta, \dot \zeta) \,
  \dd t -1\\
 &\geq&  S(y_1, \zeta(t_1)) + S(\zeta(t_1), \zeta(t_2)) +
 S(\zeta(t_2),y_2) -1  \\ &\geq&
2 \, P + 2  - P -1 =  P+1,
\end{eqnarray*}
which is in contrast with the very definition of $P$.

\end{proof}

\medskip

\subsection{Existence of special subsolutions and solutions} \;

Here we show the existence of  bounded  critical subsolutions, and
of  coercive critical solutions.

\medskip

\begin{Proposition}\label{Hkuno} There exists a  bounded  Lipschitz--continuous critical
subsolution $u$, vanishing and strict outside the compact set $K_0$
defined as in \eqref{dolo40}.
\end{Proposition}

\begin{proof}   By Proposition \ref{cortre}, item {\bf (iii)}
\begin{equation}\label{Hkuno1}
   S(y_1,y_2) \geq 0 \qquad\hbox{for any $y_1$, $y_2$ in $\ov{\R^M \setminus
K_0}$,}
\end{equation}
and consequently  the null function is an admissible trace for subsolutions to
$H_0=c_0$ on $\ov{\R^M \setminus K_0}$ in the sense of Proposition \ref{subsubsub} {\bf (iii)},
so that owing to Proposition \ref{subsubsub} {\bf (iii)}
\[u(y) := \inf\{ S(z,y) \mid z \in \ov{\R^M \setminus
K_0}\} \] is a subsolution to $H_0=c_0$ in $\R^M$  vanishing on
$\ov{\R^M \setminus K_0}$, in addition
\[ H_0(y,Du)= H_0(y,0) < c_0 - Q_0 \qquad\hbox{for $ y \in \R^M
\setminus K_0 \subset \R^M \setminus C$}\] by the very definition of
$C$ in \eqref{dolo1}.  Since  $u$ is locally Lipschitz--continuous
by Lemma  \ref{coerci} and vanishes outside a compact set, it is
actually  globally Lipschitz--continuous in $\R^M$. This fully shows
the assertion.

\end{proof}

\medskip

We denote by $\A_0$ the Aubry set of $H_0$, see Proposition \ref{OhAubry}  for the definition.
We have:

\smallskip

\begin{Lemma} \label{lemHkdue} The Aubry set $\A_0$ is nonempty and
contained in $K_0$, where $K_0$ is defined as in \eqref{dolo40}.
\end{Lemma}

\begin{proof}  We  know from Proposition \ref{Hkuno}
that there is a critical subsolution  which is strict outside $K_0$,
so that by Proposition \ref{OhAubry} {\bf (iii)}  $\A_0 \subset
K_0$. The point is then to show that the Aubry set is nonempty.

We argue by contradiction using a covering argument. If $\A_0 =
\emptyset$, then we can associate by Proposition \ref{OhAubry} {\bf
(iii)}  to any point $y \in K_0$ an open neighborhood $B_y$, a value
$d_y < c_0$, and a critical subsolution $w_y$ with
\[H_0(\cdot,Dw_y) \leq d_y < c_0  \quad\hbox{in $B_y$.}\]
We extract a finite subcovering  $\{B_1, \cdots, B_m\}$
corresponding to points $y_1, \cdots, y_m$  of $K_0$, and set
\begin{eqnarray*}
  w_j &=& w_{y_j} \\
  d_j &=& d_{y_j} \qquad{for\quad j=1, \cdots, m}.
\end{eqnarray*}
 Then
\[ \{ B_0, B_1, \cdots, B_m\}, \]
where $B_0 = \R^M \setminus K_0$, is an finite open cover of $\R^M$.
We denote by $u$ the critical subsolution constructed in Proposition
\ref{Hkuno} and set $d_0= c_0 - Q_0$, so that
\[H_0(y,Du(y)) \leq d_0 < c_0 \qquad \hbox{for  any $y \in
B_0$.}\]  We  define
\[w =  \lambda_0 \, u + \sum_{i=1}^m \lambda_j \, w_j,\]
 where $\lambda_0,\lambda_1, \cdots, \lambda_m$ are positive coefficients
summing to $1$. We have by convexity of $H_0$
 \[ H_0(y,Dw(y)) \leq  \la_0 \, H_0(y,Du(y))
+ \sum_{j=1}^m \lambda_j \, H_0(y,Dw_j(y)),\] for a.e. $y \in \R^M$,
and we derive
\[ H_0(y,Dw(y)) \leq \sum_{i\neq j}\lambda_i \, c_0 + \lambda_j \, d_j =
(1 - \lambda_j) \, c_0 + \lambda_j \, d_j= c_0 + \lambda_j \, (d_j -
c_0)\] for a.e. $ y \in B_j$, $j=0, \cdots, m$. We set $\widetilde d
= \max_j \lambda_j \, (d_j - c_0) < 0 $ and conclude
\[H_0(y,Dw(y)) \leq  c_0 + \widetilde d < c_0 \quad\hbox{for a.e. $y \in \R^M$,}\]
which is impossible by the very definition of $c_0$.  This gives  by
contradiction $\emptyset \neq \A_0 \subset K_0$, as desired.

\end{proof}

\medskip

From the previous lemma and Proposition \ref{cortre}, item {\bf (i)}
we get:

\medskip

\begin{Proposition}\label{Hkdue}
All the functions $y \mapsto S(y_0,y)$, for $y_0 \in \A_0$, are
coercive critical solutions.
\end{Proposition}

\medskip

The previous line of reasoning can be somehow reversed.  We proceed
showing that the existence of coercive solutions, plus the
coercivity of intrinsic distance, characterizes the critical
equation and also directly implies that the Aubry set is nonempty,
as made precise by the following result:

\smallskip

\begin{Proposition}\label{singa} Assume that the equation
\[H_0(y,Du)= b\]
admits a coercive solution in $\R^M$ and limit relation
\eqref{dolo6} holds true with $S_b$ in place of $S$, then $b= c_0$
and the corresponding Aubry set is nonempty.
\end{Proposition}
\begin{proof} The argument is by contradiction. Let $w$ be a coercive solution of the equation in
object.   If $b \neq c_0$ or $\A_0 = \emptyset$ then by Corollary
\ref{corohAubry}, Proposition \ref{postohAubry}, there is, for any
$R
>0$,  an unique solution   of the Dirichlet  problem
\[\left \{\begin{array}{cc}
    H_0(y,Du)=  & b \qquad\hbox{in $B(0,R)$} \\
    u = &  \; \; \; \, w  \qquad\hbox{on $\partial B(0,R)$}\\
  \end{array} \right . \]
  which therefore  must coincide with $w$, and
  \begin{equation}\label{singa1}
   w(0)= w(z) + S_b(z,0)  \qquad\hbox{for any $R >0$, some $z \in \partial B(0,R)$.}
\end{equation}
Since we have assumed  \eqref{dolo6}, with $S_b$ in place of $S$, we
have
\[\lim_{|z| \to + \infty} S_b(z,0)= + \infty\]
and by assumption $w$ is coercive. This shows that  \eqref{singa1}
is impossible, and concludes the proof.
\end{proof}

\smallskip

We derive:

\begin{Proposition} The effective Hamiltonian $\ov H: \R^N \times \R^N \to
\R$ is continuous in both components and convex in $p$.
\end{Proposition}
\begin{proof} It is easy to see using the continuity of $H$ and the argument in the
proof of Proposition \ref{critic} that $\ov H$ is locally bounded.
We consider a sequence $(x_n,p_n)$ converging to some $(x,p)$, and
assume that $\ov H(x_n,p_n)$ admits limit. We consider a sequence
$v_n$ of solutions to
\[H(x_n,y,p_n,Du) = \ov H(x_n,p_n)\]
of the form as in Proposition \ref{Hkdue}. By exploiting the
continuity of $H$ we see that the $v_n$ are locally
equiLipschitz--continuous, locally equibounded and equicoercive.
They are consequently locally uniformly convergent, up to a
subsequence, by Ascoli Theorem,  with limit function, say $w$,
locally Lipschitz-- continuous and coercive. In addition,  by basic
stability properties of viscosity solutions theory, $w$ satisfies
\[H(x,y,p,Dw)= \lim_n \ov H(x_n,p_n),\]
which  implies by Proposition \ref{singa} that $\lim_n \ov
H(x_n,p_n)= \ov H(x,p)$. This shows the claimed continuity of $\ov
H$.

We see by the very definition of $H$ that
\[H(x,y,\la \, p_1 + (1- \la) \, p_2, \la \, q_1 + (1- \la) \, q_2)
\leq  \la \, H(x,y,p_1,q_1)+ (1- \la) \,  H(x,y,p_2,q_2).\] We
derive from this that if $u_i$, $i=1,2$, satisfy $H(x,y,p_i,Du_i)
\leq \ov H(x,p_i)$ in  the viscosity sense, then
\[ H(x,y,\la \, p_1 + (1- \la) \, p_2, \la \, Du_1 + (1- \la) \, D u_2)
\leq  \la \, \ov H(x,p_1)+ (1- \la) \,  \ov H(x,p_2),\] which in
turn implies
\[\ov H(x, \la \, p_1 + (1- \la) \, p_2) \leq \la \, \ov H(x,p_1)+ (1- \la) \,  \ov
H(x,p_2)\] as desired.

\end{proof}

\medskip

\subsection{Construction of a supersolution} \; \label{construct}
 We sill keep $(x_0,p_0)$ fixed. Starting from Proposition  \ref{Hkdue}, we construct
  a supersolution
 of the cell problem which will play the role of corrector in Theorem \ref{mainuno}.
 We denote by $K_0$ the set defined in \eqref{dolo40}. We fix $y_0 \in \A_0$; by  the coercivity
 of $S(y_0,\cdot)$, see Proposition \ref{Hkdue}, there is a constant $d$ such that
 \begin{equation}\label{constrasse0}
    d + S(y_0,y) >0 \qquad\hbox{for any $y \in \R^M$.}
\end{equation}
We select a constant $R_0$ satisfying
\begin{eqnarray}
   &B(0,R_0 - 3 ) \supset K_0& \label{constrassu1}  \\
   &\hbox{$R_0 - 3$ satisfies Corollary \ref{corquattro} for a neighborhood of
   $y_0$.}& \label{constrassu11}\end{eqnarray}

\medskip

 We aim at proving:

\smallskip

\begin{Theorem}\label{constrmain} Let  $U: \R^M \to \R$ be  a function   bounded
 from above  in $\ov {B}(0,R_0)$ with
\begin{equation}
  U \leq  0 \qquad\hbox{in $B(0,R_0-1)$, } \label{constrassu2}
\end{equation}
 then there exists for
 any $\la >0$,   a  locally Lipschitz--continuous supersolution
$w_\la$ of $H_0 = c_0$  in $\R^M$ with
\begin{eqnarray}
  U &\leq& \la \, w_\la \,\qquad\quad\qquad\hbox{in $\ov{B}(0,R_0)$} \label{constrasse1} \\
  w_\la &=& d + S(y_0, \cdot)  \,\;\qquad\hbox{in a neighborhood of $y_0$.} \label{constrasse2}
\end{eqnarray}
\end{Theorem}
\medskip

To construct the supersolutions $w_\la$ some preliminary steps are
needed.

We define
\[ M_0= \max \left \{ \sup_{\ov{B}(0,R_0)} \frac 1\la \, U, 1 \right \}.\]
We denote by $h_\la:[0,+ \infty) \to [0,+\infty)$ a nondecreasing
continuous function with
\begin{eqnarray}
  h_\la &\equiv& 1 \quad\qquad\hbox{in $[0,R_0-3]$}  \label{constr010}\\
  h_\la &\equiv&  M_0 \,\qquad\hbox{in $[R_0 - 2, + \infty)$.}\label{constr012}
  \end{eqnarray}
  We introduce  the length functional
   \[\int_0^1 h_\la( |\xi|) \, \si(\xi, \dot\xi) \, ds \]
 for any curve $\xi$ defined in $[0,1]$,  and denote by $S^h$ the distance
 obtained by minimization of it among curves linking two given points, we drop dependence on $\la$
 to ease notations.

 \smallskip

 \begin{Lemma}\label{constru2015} The function $S^h(y_0,\cdot)$ is a locally
 Lipschitz--continuous supersolution to $H_0=c_0$ in $\R^M$, and coincides with
 $S(y_0,\cdot)$ in a neighborhood of $y_0$.
\end{Lemma}
\begin{proof} The function  $h_\la$, defined in  \eqref{constr010},
\eqref{constr012}, satisfies $h_\la \geq 1$ and if $h_\la(|y|) >1$
then by \eqref{constr010}
\[y \not\in B(0,R_0-3) \supset K_0\]
  so that  by  Proposition \ref{cortre} {\bf (ii)} $H_0(y,0) < c_0$. We are thus in position to
apply Proposition \ref{length}, which directly gives the asserted
supersolution property outside $y_0$, as well as the Lipschitz
continuity.  We also know by \eqref{constrassu11} and $h_\la \equiv
1$ in $B(0,R_0-3)$ that
\[ S^h(y_0,\cdot)= S(y_0,\cdot) \qquad\hbox{in a neighborhood of $y_0$,}\]
and  $S^h(y_0,\cdot)$ is solution to $H_0=c_0$ on the whole space,
by Proposition \ref{Hkdue}. This concludes the proof.
\end{proof}

\smallskip

By the very definition of $S^h$, we have:

\begin{equation} \label{constr2003}
  S^h \geq  S \;\quad\qquad\hbox{in $\R^M \times \R^M$.}
\end{equation}

\smallskip

We  define

\begin{equation}\label{constrdef}
    w_\la = d + S^h(y_0, \cdot)
\end{equation}
where $d$, $y_0$ are as in  \eqref{constrasse0}.

\smallskip

\begin{Lemma}\label{constr3} The following inequalities hold true:
\begin{eqnarray*}
  w_\la &>& 0  \quad\qquad\hbox{in $\R^M$} \\
  w_\la &\geq& M_0    \, \qquad\hbox{in $\R^M \setminus B(0,R_0-1)$.}
\end{eqnarray*}
\end{Lemma}
\begin{proof} From \eqref{constr2003}  and the
definition of $w_\la$ we derive
\[w_\la \geq d + S(y_0,\cdot)\]
and this  in turn yields $w_\la > 0$ in  $\R^M$ because of
\eqref{constrasse0}.

 We fix $y \not\in B(0,R_0 -1)$,  and
consider any  curve $\zeta$ defined in $[0,1]$ linking $y_0$ to $y$.
We set
\[ t_1 =  \max \{t \in [0,1] \mid \zeta(t) \in B(0,R_0-2)\},\]
notice that
\[ |\zeta(t_1) - y| > 1.\]
Owing to the above inequality, $w_\la >0$,  Proposition \ref{cortre}
item {\bf (ii)}, the definition of $h_\la$,  we have
\begin{eqnarray*}
 d + \int_{0}^{1} h_\la(|\zeta|) \, \si(\zeta, \dot \zeta)
  \, \dd t &=& d +
 \int_0^{t_1} h_\la(|\zeta|) \, \si(\zeta, \dot \zeta)
  \, \dd t+ \int_{t_1}^{1} h_\la(|\zeta|) \, \si(\zeta, \dot \zeta)
  \, \dd t \\ &\geq&  w_\la(\zeta(t_1)) + \int_{t_1}^{1} h_\la(|\zeta|) \,
  |\dot\zeta| \, \dd t  \\ &\geq&  w_\la(\zeta(t_1))  + M_0 \, |y -
  \zeta(t_1)|
  > M_0.
\end{eqnarray*}
Taking into account the definition of $w_\la$ and the fact that the
curve $\zeta$ joining $y_0$  to $y \not\in B(0,R_0 -1)$  is
arbitrary, we deduce from the above computation the desired
inequality.
\end{proof}

\medskip

\begin{proof} {\bf (of Theorem \ref{constrmain})} \; In view of
Lemma \ref{constru2015}, it is just left to show
\eqref{constrasse1}. It indeed holds true in $B(0,R_0 -1)$ because
of \eqref{constrassu2} and $w_\la >0$.  If $y \in \ov{B(0,R_0)}
\setminus B(0,R_0-1)$, then by Lemma \ref{constr3}, we have
\[ w_\la(y) \geq M_0 \geq  \sup_{\ov{B}(0,R_0)} \frac 1\la \, U \geq
\frac 1\la \, U(y).\]

\end{proof}

\bigskip

\section{Asymptotic analysis}

  \parskip +3pt

  We summarize the relevant output of the previous  section in the
  following

  \smallskip

  \begin{Theorem} We consider  $(x_0,p_0) \in \R^N \times \R^N$, a constant $R_0$ satisfying
  \eqref{constrassu1}, \eqref{constrassu11}, a function $U$ bounded from above in $\ov B(0,R_0)$
  and less than or equal to zero in $B(0,R_0-1)$, any positive constant $\la$. Then the equation
  \[H(x_0,y,p_0,Du)= \overline H(x_0,p_0)  \qquad\hbox{in $\R^M$}\]
  admits a bounded Lipschitz--continuous subsolution
  and a  locally Lipschitz--continuous supersolution, say $w_\la$,
  satisfying \eqref{constrasse1}, \eqref{constrasse2}
\end{Theorem}

\medskip

We recall the notations  $\ov V= {\limsup}^\# V^\eps$, $\underline
V=  {\liminf}_\# V^\eps$,  where the $V^\eps$ are the value
functions of problems \eqref{minprob}/ \eqref{minprobbis}. We
consider a point $(x_0,t_0) \in \R^N \times (0,+ \infty)$, and  set
\begin{equation}\label{Kde}
    K_\de = B(x_0,\de) \times (t_0-\de,t_0+\de) \qquad\hbox{for $\de <
t_0$.}
\end{equation}
 We further consider   a constant $R_0 >0$ satisfying
\eqref{constrassu1}, \eqref{constrassu11}.
 The next lemma, based on Theorem \ref{constrmain}, will be of crucial importance.
The entities  $y_0 \in \A_0$ and $d$ appearing in the statement are defined as in \eqref{constrasse0} :

\medskip

\begin{Lemma}\label{cruci} Let $\psi$ be     a strict supertangent
 to $\ov V$ at $(x_0,t_0)$ such that  $(x_0,t_0)$ is the unique maximizer of $\ov V - \psi$
  in $K_{\de_0}$, for some $\de_0<t_0$.  Then, given any infinitesimal sequence
   $\eps_j$, and $\de < \de_0$,
we find
   a constant $\rho_\de$
  and a   family $w^j$ of supersolutions to $H(x_0,y,D\psi(x_0,t_0),Du)= \ov
H(x_0,D\psi(x_0,t_0))$  in $\R^M$ satisfying for $j$ suitably large
\begin{eqnarray}
  \eps_j \, w^j &\geq& V^{\eps_j} - \psi + \rho_\de \qquad\hbox{in
$\partial \big (K_\de \times B(0,R_0) \big )$}  \label{cruci001}\\
  w^j &=& d + S(y_0,\cdot) \qquad\hbox{ in a neighborhood $A_0$ of
  $y_0$,} \label{cruci002}
\end{eqnarray}
where $S$ is the intrinsic critical distance, see Subsection
\ref{basic}, related to $(x_0,D\psi(x_0,t_0))$.
\end{Lemma}

\begin {proof}
By supertangency properties of $\psi$ at $(x_0,t_0)$, we find, for
any $\de < \de_0$, a $\rho_\de
> 0 $ with
\begin{equation}\label{cruci1}
    \max_{\partial K_\de} \big ( \ov V - \psi \big ) < - 3 \, \rho_\de.
\end{equation}
We  fix a $\de$ and define
\[U^\eps(y) = \left \{ \begin{array}{cc}
    \max_{(x,t) \in \partial K_\de} \left \{V^{\eps}(x,y,t) - \psi(x,t) +
  \rho_\de \right \} & \,\quad\hbox{for $y \in B(0,R_0-1/2)$} \\
    \max_{(x,t) \in K_\de}  \left \{V^{\eps}(x,y,t) -
  \psi(x,t) + \rho_\de \right \} & \qquad\qquad\hbox
  {for $y \in \R^M \setminus B(0,R_0-1/2)$.} \\
  \end{array} \right .\]
Notice that the $U^\eps$ are continuous  for any $\eps$ and locally
equibounded, since the $V^\eps$ are locally equibounded in force of
Proposition \ref{estimo}. To ease notations we set

\[U^j = U^{\eps_j}.\]
\smallskip

\noindent {\bf Claim :} {\em There is $j_0=j_0(R_0)$ such that
\[U^j \leq - \rho_\de \qquad\hbox{in $B(0,R_0-1)$, for $j > j_0$.}\]}

\smallskip
Were the claim false, there should be  a subsequence
 $y_j$ contained in $B(0,R_0-1)$  with
\[U^j(y_j) > - \rho_\de.\]
 The $y_j$  converge, up to further extracting a subsequence,
  to some  $\ov y$, and, being $\eps_j$ infinitesimal, we get
\begin{equation}\label{constr10}
    ({\limsup}^\# U^\eps)(\ov y) \geq -\rho_\de.
\end{equation}
Moreover,  there exists an infinitesimal sequence $\eps_i$ and
elements $z_i$ converging to $\ov y$ with
\[\lim_i U^{\eps_i}(z_i)= ({\limsup}^\# U^\eps)(\ov y),\]
at least for $i$ large $z_i \in B(0,R_0-1/2)$, and by the very
definition of $U^\eps$ in $B(0,R_0-1/2)$, we get
\[ U^{\eps_i}(z_{i}) = V^{\eps_i}(x_i,z_i,t_i)
 - \psi(x_i,t_i) + \rho_\de
\qquad\hbox{for some $(x_i,t_i) \in \partial K_\de$,}\] up to
extracting a subsequence, $(x_i, t_i)$ converges to some $(\ov x,\ov
t) \in \partial K_\de$ so that by \eqref{cruci1}
\begin{eqnarray*}
  ({\limsup}^\# U^\eps)(\ov y) &=& \lim U^{\eps_i}(z_i) =
  \lim \big [ V^{\eps_i}(x_i,z_i,t_i)
  - \psi(x_i,t_i) + \rho_\de \big ]\\
   &\leq& \ov V(\ov x,\ov t) - \psi(\ov x,\ov t)+ \rho_\de  \leq - 2 \, \rho_\de.
\end{eqnarray*}
which is in contradiction with \eqref{constr10}. This ends the proof
of the claim.

\smallskip

 We are then in the position to  apply Theorem
\ref{constrmain} to  any $U^j$, and get a  supersolution $w^j$ to
$H(x_0, \cdot,D\psi(x_0,t_0),\cdot)=  \ov H(x_0,D \psi(x_0,t_0))$,
which satisfies, for $j
> j_0$,  the condition  \eqref{cruci002} and
\[ \eps_j \, w^j  \geq   U^j  \qquad\hbox{in $\ov{B}(0,R_0)$.}\]
Owing to the very definition of $U^j$, we derive from the latter
inequality that
\[ \eps_j \, w^j(y)\geq V^{\eps_j}(x,y,t) - \psi(x,t) + \rho_\de \]
holds in
\[\partial K_\de
\times B(0,R_0) \cup K_\de \times \partial B(0,R_0)= \partial \big (K_\de \times B(0,R_0) \big ).\]
This proves   \eqref{cruci001} and conclude the proof.

\end{proof}

We proceed establishing  the asymptotic result for upper weak
semilimit of the $V^\eps$. The first part of the proof is a version,
adapted to our setting, of perturbed test function method. We are
going to use as correctors, depending on $\eps$, the special
supersolutions to cell equations constructed in Subsection
\ref{construct} in the frame of Lemma \ref{cruci}. The argument of
the second half about behavior of limit function at $t=0$ makes a
direct use of  the material of Subsections \ref{subdyna},
\ref{subvalue}.

\smallskip

\begin{Theorem}\label{mainuno} The function $\ov V={\limsup}^\# V^\eps$ is a subsolution to
 \eqref{Hjlim} satisfying
 \begin{equation}\label{mainuno000}
    \limsup_{(x,t) \to (x_0,0) \atop t >0} \ov V(x,t) \leq \ov u_0(x_0)
\qquad  \hbox{for any $x_0 \in \R^N $.}
\end{equation}
\end{Theorem}

\begin{proof} Let $(x_0,t_0)$ be a point  in $\R^N \times (0, + \infty)$, and $\psi$
a strict supertangent to $\ov V$ at $(x_0,t_0)$ such that
$(x_0,t_0)$ is the unique maximizer of $\ov V - \psi$ in
$K_{\de_0}$, for some $\de_0
>0$ (see \eqref{Kde} for the definition of $K_\de$).

 By Proposition \ref{perez}, we can find an infinitesimal sequence $\eps_j$
 and $(x_j,y_j,t_j)$ converging to $(x_0,y_0,t_0)$, where $y_0$
is as in \eqref{constrasse0}, with
\begin{equation}\label{mainuno1}
    \lim_j V^{\eps_j}(x_j,y_j, t_j) = \ov V(x_0,t_0)= \psi(x_0,t_0).
\end{equation}
We assume by contradiction
\begin{equation}\label{mainuno11}
    \psi_t(x_0,t_0) + \ov H(x_0,D \psi(x_0,t_0)) > 2 \, \eta
\end{equation}
for some positive $\eta$.  We apply Lemma \ref{astima}, about
coercivity of $H$, to the bounded set
\[C:= B(x_0,\de_0) \times B(0,R_0) \times D\psi(K_{\de_0}),\]
where $R_0$ satisfies \eqref{constrassu1}, \eqref{constrassu11}, and
exploit that $\ov H$ is locally bounded to find $P > 0$ with
\begin{equation}\label{mainuno01}
   H(x,y,p,q) > \ov H(x,p) \qquad\hbox{for $(x,y,p) \in C$, $q$ with
$ |q| \geq P$.}
\end{equation}
Applying the estimates   \eqref{compara} to $B(x_0,\de_0) \times
B(0,R_0)$ and \eqref{comparabis}, we find
\begin{eqnarray}
 && |H(x_0,y,D\psi(x_0,t_0),q) -H(x,y,p,q)| \leq  \label{mainuno51} \\
 && L_0 \, (|x-x_0| ) (|D\psi(x_0,t_0)| + |q|) +  \nonumber\\
 && \om(|x-x_0|) + Q \, |D \psi(x_0,t_0) -p| \nonumber
\end{eqnarray}
 for any $(x,y) \in B(x_0,\de_0) \times B(0,R_0)$ and  $(p,q) \in  \R^N \times \R^M$, where
 $\om$ is an uniform continuity modulus of $\ell$ in $B(x_0,\de_0) \times B(0,R_0) \times
 A$,  $L_0$ is as in {\bf (H2)} and $Q$ is an upper bound of $|f|$
 in $B(x_0,\de_0) \times B(0,R_0) \times A$.

Exploiting  the continuity of $D \psi$, $\psi_t$, $\ov H$,  we can
determine, $\de_0 > \de > 0$  such that using  \eqref{mainuno11},
\eqref{mainuno51} with $q \in B(0,P)$ and $p$ of the form $D
\psi(x,t)$, we get
\begin{eqnarray}
  |H(x_0,y,D \psi(x_0,t_0),q) - H(x,y,D \psi(x,t),q)| &<& \eta \label{mainuno2}\\
  |D\psi(x,t) - D\psi(x_0,t_0)| &<& \eta \label{mainuno2bis}\\
  \psi_t(x,t) + \ov H(x,D \psi(x,t)) &>& 0 \label{mainuno2tris}
\end{eqnarray}
for $(y,q) \in B(0,R_0) \times B(0,P)$, $(x,t) \in K_\de$. By
applying Lemma \ref{cruci} to such a $\de$, we  find a constant
$\rho_\de
>0$ and a family $w^j$ of supersolutions to
\[H(x_0,y,p_0,D\psi(x_0,t_0),Du)= \ov H(x_0, D\psi(x_0,t_0)) \qquad\hbox{in $\R^M$}\]
 with
\begin{eqnarray}
  \eps_j \, w^j &\geq& V^{\eps_j} - \psi + \rho_\de \qquad\hbox{in
$\partial \big (K_\de \times B(0,R_0) \big )$} \label{mainuno20} \\
  w^j &=& d + S_0(y_0,\cdot) \qquad\hbox{ in a neighborhood $A_0$ of
  $y_0$,} \label{mainuno21}
\end{eqnarray}
for $j$ large enough, see \eqref{constrasse0} for the definition of
$d$. We claim that the corrected test function $\psi +   w^j$
satisfies
\[  \psi_t(x,t) + H (x,y,D \psi(x,t), D w^j) \geq 0\]
in $K_\de \times B(0,R_0)$ in the viscosity sense. In fact, let
$\phi$ be a subtangent to $\psi +   w^j$ at some point  $(x,y,t) \in
K_\de \times B(0,R_0)$, then
\begin{eqnarray*}
  \phi_t(x,y,t) &=& \psi_t(x,t) \\
  D_x \phi(x,y,t) &=& D\psi(x,t)
\end{eqnarray*}
and so, to prove the claim, we have to show the inequality
\[ \psi_t(x,t) + H (x,y,D \psi(x,t), D_y \phi(x,y,t)) \geq 0.\]
We have that
\[ z \mapsto \phi(x,z,t)\]
is supertangent to  $w^j$ at $y$, which implies by the supersolution
property of $w^j$
\[H (x_0,y,D \psi(x_0,t_0), D_y \phi(x,y,t)) \geq \ov H(x_0,D\psi(x_0,t_0))\]
If $|D_y \phi(x,y,t)| < P$ then by \eqref{mainuno11},
\eqref{mainuno2} and \eqref{mainuno2bis}
\begin{eqnarray*}
  &\psi_t(x,t) + H (x,y,D \psi(x,t), D_y \phi(x,y,t)) \geq&  \\
   & \psi_t(x_0,t_0) - \eta  + H (x_0,y,D \psi(x_0,t_0), D_y \phi(x,y,t)) - \eta
   \geq&\\
   &\psi_t(x_0,t_0) + \ov H(x_0,D\psi(x_0,t_0)) - 2 \, \eta \geq 0.
\end{eqnarray*}
If instead $|D_y \phi(x,y,t)| \geq P$ then by \eqref{mainuno01},
\eqref{mainuno2tris}
\begin{eqnarray*}
  &\psi_t(x,t) + H (x,y,D \psi(x,t), D_y \phi(x,y,t)) \geq&  \\
   & \psi_t(x,t)   + \ov H(x,D\psi(x,t)) \geq 0.
\end{eqnarray*}
The claim is then proved. For $j$ large enough, the functions
$V^{\eps_j}$, $\psi + \eps_j \, w^j - \rho_\de$ are then
subsolutions and supersolutions, respectively, to
\[u_t + H \left (x,y,D_x u, \frac{D_y u}{\eps_j} \right )=0\]
in $K_\de \times B(0,R_0)$, then taking into account the boundary
inequality \eqref{mainuno20}, we can apply the comparison principle
of Proposition \ref{comparacompara} to the above equation to deduce
\begin{equation}\label{mainuno30}
     V^{\eps_j}  \leq \psi + \eps_j \, w^j - \rho_\de\qquad\hbox{in $K_\de \times
B(0,R_0)$.}
\end{equation}
 On the other side, let $(x_j,y_j,t_j)$ be the sequence converging to $(x_0,y_0,t_0)$
 introduced in \eqref{mainuno1}, then  for $j$ large $(x_j,y_j,t_j) \in K_\de \times B(0,R_0)$,
  and  $w^j(y_j)= d + S(y_0,y_j)$  by \eqref{mainuno21}, so that
 \[\lim_j \eps_j \, w^j(y_j)= 0.\]
 We therefore get
\[ \lim_j \big [ V^{\eps_j}(x_j,y_j,
t_j)-\psi(x_j,t_j) - \eps_j \, w^j(y_j) \big ] = \ov V(x_0,t_0)-
\psi(x_0,t_0)=0\] which contradicts \eqref{mainuno30}.

We proceed proving \eqref{mainuno000}. We consider $(x_n, t_n)$
converging to $(x_0,0)$ such that $\ov V(x_n,t_n)$ admits limit. Our
task is then to show
\[\lim_n \ov V(x_n,t_n) \leq \ov u_0(x_0).\]
We find   for any $n$ an infinitesimal sequence $\eps^n_j$ and
$(x^n_j, y^n_j,t^n_j)$ converging to $(x_n,0,t_n)$ with
\[\lim_j V^{\eps^n_j}(x^n_j, y^n_j,t^n_j)= \ov V(x_n,t_n),\]
$0 \in \R^M$ is clearly an arbitrary choice, in view of Proposition
\ref{perez}. By applying a diagonal argument we find $\eps_n$
converging to $0$ and $(z_n,y_n,s_n)$ converging to $(x_0,0,0)$ with
\begin{eqnarray}
  \lim_n V^{\eps_n}(z_n, y_n,s_n) &=& \lim_n \ov V(x_n,t_n) \label{maindue99} \\
  \lim_n \frac {s_n}{\eps_n} &=& + \infty. \label{maindue100}
\end{eqnarray}
Given $\de > 0$, we denote by $\widetilde y$ a $\de$--minimizer of
$y \mapsto u_0(x_0,y)$ in $\R^M$, see assumption {\bf (H5)}. By
applying Proposition \ref{straestimo},  Lemma \ref{prestimobis} and
taking into account \eqref{maindue100}, we find for any $n$
sufficiently large a trajectory $(\xi_n,\eta_n)$ of \eqref{dynabis},
with $\eps=\eps_n$, corresponding to controls $\al_n$ and  starting
at $(z_n,y_n)$, such  that
\begin{eqnarray}
   &(\xi_n, \eta_n) \; \hbox{is contained in a compact
    subset independent of $n$ as $t \in [0,s_n/\eps_n]$}& \label{maindue101}  \\
   &|\eta_n(s_n/\eps_n) - \widetilde y|= \OO (\eps_n)& \label{maindue102}
\end{eqnarray}
 By using formulation \eqref{minprob} of minimization problem, we
discover
\[V^{\eps_n}(z_n, y_n,s_n) \leq \eps_n \, \int_0^{\frac {s_n}{\eps_n}} \ell( \xi_n(t),
\eta_n(t), \al_n(t)) \, \dd t + u_0(\xi_n(s_n/\eps_n),
\eta_n(s_n/\eps_n)),\] where the integrand is estimated from above
by a constant, say $Q$, independent of $n$, because of
\eqref{maindue101}, therefore
\[ V^{\eps_n}(x_n, y_n,s_n) \leq Q \, s_n +  u_0(\xi_n(s_n/\eps_n),
\eta_n(s_n/\eps_n))\] Owing to \eqref{est1}, \eqref{maindue102},
\eqref{maindue99}, and the fact that $s_n$ is infinitesimal, we then
get
\[\lim_n \ov V(x_n,t_n)=\lim_n V^{\eps_n}(z_n, y_n,s_n) \leq u_0(x_0, \widetilde y)
\leq \ov u_0(x_0) + \de.\]
This concludes the proof because $\de$ is arbitrary.
\end{proof}

\medskip

The second main result concerns lower weak semilimit. Here we
essentially exploit the existence of bounded Lipschitz--continuous
subsolutions to  cell equations established in Proposition
\ref{Hkuno} plus the coercivity of the $V^\eps$ proved in
Proposition \ref{estimobis}. The part of the proof about behavior of
limit function at $t=0$ is direct and not based on a PDE approach.
We recall that $(\ov u_0)_\#$ stands for the lower semicontinuous
envelope of $\ov u_0$, see Subsection \ref{notations} for
definition.

\smallskip

\begin{Theorem}\label{maindue} The function  $\underline V={\liminf}_\# V^\eps$ is a
supersolution to
 \eqref{Hjlim} satisfying
 \begin{equation}\label{mainuno000bis}
    \liminf_{(x,t) \to (x_0,0) \atop t >0} \underline V(x,t) \geq (\ov u_0)_\#(x_0)
\qquad  \hbox{for any $x_0 \in \R^N $.}
\end{equation}
\end{Theorem}

 \begin{proof}
Let $(x_0,t_0)$ be a point  in $\R^N \times (0, + \infty)$, and
$\varphi$ a strict subtangent to $\underline V$ at $(x_0,t_0)$ such that
$(x_0,t_0)$ is the unique minimizer of $\underline V - \varphi$ in
$K_{\de_0}$, for some $\de_0
>0$ (see \eqref{Kde} for the definition of $K_\de$).  We
 assume by contradiction
\begin{equation}\label{maindue11}
    \varphi_t(x_0,t_0) + \ov H(x_0,D \varphi(x_0,t_0)) < 0.
\end{equation}
  Given  $\eps >0$, we can find by
Proposition \ref{estimobis} about coercivity of value functions,
$R_\eps >1$ satisfying
\begin{equation}\label{maindue2}
  V^\eps(x,y,t) > \sup_{K_{\de_0}} \varphi + 1
\qquad\hbox{for $(x,t) \in K_{\de_0}$, $y \in \R^M \setminus
B(0,R_\eps)$.}
\end{equation}
We can also find, exploiting Proposition \ref{Hkuno},
 a Lipschitz--continuous subsolution $u$ to the cell problem
 \begin{equation}\label{maindue3}
    H(x_0,y,D\varphi(x_0,t_0),Du) = \ov H(x_0,D \varphi(x_0,t_0)) \qquad\hbox{in $\R^M$}
\end{equation}
with
\begin{equation}\label{maindue4}
   u(y) < 0 \, \qquad\hbox{ for any $y \in \R^M$.}
\end{equation}
By using   estimate \eqref{compara} on $H$, Lipschitz continuity of
$u$, continuity of $\ov H$, $D \varphi$, $\varphi_t$ and
\eqref{maindue11}, \eqref{maindue3} we can determine $0 < \de <
\de_0$ such that $u + \varphi $ is subsolution to
\[w_t + H(x,y,D\varphi(x,t),Dw) = 0   \qquad\hbox{in $K_\de \times \R^M$.}\]
Owing to strict subtangency property of $\varphi$, there is $1 >
\rho >0$ with
\[\underline V - \varphi > 2 \, \rho \qquad\hbox{in $\partial K_\de$,}\]
and, taking into account that $\underline V$ is the lower semilimit
of the $V^\eps$, we derive
\[V^\eps - \varphi > \rho \qquad\hbox{in $\partial K_\de \times
B(0,R_\eps)$}\] for $\eps$ sufficiently small,  which  in turn
implies by \eqref{maindue4}
\begin{equation}\label{maindue5}
    V^\eps - \varphi - u > \rho \qquad\hbox{in $\partial K_\de \times
B(0,R_\eps)$.}
\end{equation}
Owing to \eqref{maindue2},  \eqref{maindue4}, we also have
\begin{equation}\label{maindue6}
    V^\eps - \varphi - u > \rho  \qquad\hbox{in $ K_\de \times \partial
B(0,R_\eps)$.}
\end{equation}
Since $V^{\eps}$, $\varphi + \eps \, u + \rho$ are  supersolution and
subsolution, respectively, to
\[w_t + H \left (x,y,D_x w, \frac{D_y w}{\eps} \right )=0\]
in $K_\de \times B(0,R_0)$, the boundary conditions
\eqref{maindue5}, \eqref{maindue6} plus  the comparison principle in
Proposition \ref{comparacompara} implies
\begin{equation}\label{maindue7}
     V^{\eps}  \geq \varphi + \eps  \, u + \rho \qquad\hbox{in $K_\de \times
B(0,R_\eps)$, for $\eps$ small.}
\end{equation}
On the other side, there is by Proposition \ref{perez} an
infinitesimal sequence $\eps_j$ and a sequence $(x_j,y_j,t_j)$
converging to $(x_0,0,t_0)$
 with
 \[\lim_j V^{\eps_j}(x_j,y_j,t_j)= \underline V(x_0,t_0)\]
 and consequently
\[ \lim_j  \big [ V^{\eps_j}(x_j,y_j,
t_j)-\varphi(x_j,t_j) - \eps_j  \, u (y_j) \big ] = \ov V(x_0,t_0)-
\varphi(x_0,t_0)=0.\]  Taking into account that $R_\eps >1$ for any
$\eps$, and  $(x_j,y_j,t_j)$ are in $K_\de \times B(0,1)$ for $j$
large, the last limit relation contradicts \eqref{maindue7}.

We proceed proving \eqref{mainuno000bis}. We consider $(x_n, t_n)$
converging to $(x_0,0)$ such that $\underline V(x_n,t_n)$ admits limit,
with the aim  of showing
\[\lim_n \underline V(x_n,t_n) \geq (\ov u_0)_\#(x_0).\]
Arguing as in  the final part of Theorem \ref{mainuno}, we find an
infinitesimal sequence  $\eps_n$  and $(z_n,y_n,s_n)$ converging to
$(x_0,\widetilde y,0)$, for some $\widetilde y \in \R^M$, with
\[\lim_n V^{\eps_n}(z_n, y_n,s_n)= \lim_n \underline V(x_n,t_n).\]
We fix $\de >0$. Arguing  as in second half of Proposition
\ref{estimo}, see estimate \eqref{estimo8}, we determine a constant
$P_0$ independent of $n$ and trajectories $(\xi_n,\eta_n)$ of the
controlled dynamics starting at $(z_n,y_n)$   with
\[V^{\eps_n}(z_n, y_n,s_n) \geq P_0 \, s_n + u_0(\xi_n(s_n/\eps_n),
\eta_n(s_n/\eps_n)) - \de \geq P_0 \, s_n + \ov
u_0(\xi_n(s_n/\eps_n)) - \de.\] Since by the boundedness assumption
on $f$
\[|\xi_n(s_n/\eps_n) - z_n| \leq  Q_0 \, s_n,\]
 we get at the limit
\[\lim_n \underline V (x_n,t_n) =\lim_n V^{\eps_n}(z_n, y_n,s_n)
\geq \liminf_n  \ov u_0(\xi_n(s_n/\eps_n)) - \de \geq (\ov
u_0)_\#(x_0) -\de,\] which gives the assertion since $\de$ is
arbitrary.

 \end{proof}

\bigskip

\appendix

\section{Facts from weak KAM theory} \label{appendice}

 \parskip +3pt

Here we consider an Hamiltonian $F(y,q)$ defined in $\R^M \times
\R^M$ and  the family of equations

\begin{equation}\label{eq}
    F(y,Du)=b \qquad\hbox{in $\R^M$, for $b \in \R$}
\end{equation}

\small We assume $F$ to satisfy

\begin{enumerate}
    \item [] $F$ is continuous in both variables;
    \item []  $F$ is convex in $q$;
    \item [] $\lim_{|q| \to + \infty} \, \min_{y \in K} F(y,q)= + \infty$
    for any compact subset $K$  of $\R^M$.
\end{enumerate}
Our aim is to recall some basic facts of weak KAM theory, which will
be exposed here  through the so--called metric method for equation
\eqref{eq}, see \cite{F1}, \cite{FS1}, \cite{FS2},  \cite{F2}. We
define the critical value of $F$ as
\begin{equation}\label{cricri}
    c= \inf \{b \mid \eqref{eq} \;\;\hbox{has subsolutions in $\R^M$}\}.
\end{equation}
Being the ambient space non compact $c$ can also be infinite. We assume in what follows
\begin{enumerate}
    \item[] The critical value of $F$ is finite.
\end{enumerate}
\medskip
We call supercritical a value $b$ with $b \geq c$.  By stability properties of viscosity
(sub)solutions,  subsolutions for the critical equation do exist. We derive from coercivity of $F$:
\smallskip

\begin{Lemma}\label{equili} Let $b$ a supercritical value.
The subsolutions to $F=b$ are locally equiLipschitz--continuous.

\end{Lemma}

We adopt the so--called metric method
which is based on the definition of an intrinsic distance starting from the sublevels
of the Hamiltonian for any supercritical value. For any $b \geq c$ we set
\[ Z_b(y)= \{q \mid F(y,q) \leq b\} \qquad\hbox{
$y \in \R^M$.}\]

\smallskip

Owing to continuity, convexity and coercivity of $F$,  we have:

\smallskip

\begin{Lemma} For any $ b \geq c$, the multifunction  $y \mapsto Z_b(y)$ takes convex compact values,
it is in addition Hausdorff--continuous at any point $y_0$ where $\interior Z_b(y_0) \neq \emptyset$
and upper semicontinuous elsewhere.
\end{Lemma}
\smallskip
 We further set
\[\si_b(y,v) = \max \{q \cdot v \mid q \in Z_b(y)\} \qquad\hbox{for  any
$y$, $v$ in $\R^M$},\]
namely the support function of $Z_b(y)$ at $q$,
and define for any curve $\xi$ defined in $[0,1]$ the associated intrinsic length via
\[\int_0^1 \si_b(\xi, \dot\xi) \, \dd s.\]
Notice that the above integral is invariant for
orientation-preserving change of parameter, being the support
function positively homogeneous and subadditive, as a  length
functional should be.  Also notice that because of this invariance
the choice of the interval $[0,1]$ is not restrictive.  For any pair
$y_1$, $y_2$ we define the intrinsic distance as
\[S_b(y_1,y_2) = \inf \left \{\int_0^1 \si_b(\xi, \dot\xi) \, \dd s \mid \xi \;\hbox{ with
$\xi(0)=y_1$, \, $\xi(1)=y_2$} \right \}.\] The intrinsic distance
is finite for any supercritical value $b$.

\smallskip

\begin{Proposition}\label{subsubsub} Given $b \geq c$, we have
\begin{itemize}
    \item [\bf (i)] a function $u$ is a subsolution to $F=b$ if and only if
    \[u(y_2)-u(y_1) \leq S_b(y_1,y_2) \qquad\hbox{for any $y_1$, $y_2$;}\]
    \item [\bf (ii)] for any fixed $y_0$, the function
     $y \mapsto S_b(y_0,y)$ is subsolution to $F=b$ in $\R^M$ and solution in
      $\R^M \setminus \{y_0\}$;
     \item [\bf (iii)] Let $C$, $w$ be a closed set of $\R^M$ and a function defined in $C$ satisfying
\[w(y_2)- w(y_1) \leq S_b(y_1,y_2) \qquad\hbox{for any $y_1$, $y_2$ in $C$}\]
then the function
\[y \mapsto  \inf \{w(z) + S_b(z,y) \mid z \in C\}\]
is subsolution to $F=b$ in $\R^M$, solution in $\R^M \setminus C$
and  equal to $w$ in $C$.
\end{itemize}
\end{Proposition}

\smallskip

In contrast to what happens when the ambient space is compact,
namely $F=b$ admits solutions in the whole space if and only if $b=c$, in the
noncompact case instead there are solutions for any supercritical equation. It is in fact
enough that the intrinsic length is finite, as always is the case for supercritical values, to get
a solution.

The construction of such a solution is in fact quite simple. One considers a  sequence $y_n$
with $|y_n|$ diverging and the functions
\[ u_n= S_b(y_n, \cdot) - S_b(y_n,0).\]
By Lemma \ref{equili} and Proposition \ref{subsubsub} the $u_n$ are
solutions except at $y_n$, are locally equiLipschitz--continuous,
and also equibounded, since they vanish at $0$. They therefore
converge, up to a subsequence, by Ascoli Theorem. Having swept away
the bad (in the sense of Proposition \ref{subsubsub} {\bf (ii)})
points $y_n$ to infinity, but kept the solution property by
stability properties of viscosity solutions under uniform
convergence, we see that the limit function is indeed the sought
solution of $F=b$.

\medskip

We say that a function $u$ is a strict subsolution to $F=b$ in some
open set $B$ if
\[F(x,Du) \leq b -\de \qquad\hbox{for some $\de >0$, in the viscosity
sense in $B$.}\]

\smallskip

The points satisfying the equivalent  properties stated in the
following proposition, make up the so--called Aubry set, denoted by
$\A$.

\smallskip

\begin{Proposition}\label{OhAubry} Given $y_0 \in \R^M$, the following three properties are equivalent:
\begin{itemize}
    \item [\bf(i)] there exists a sequence of cycles $\xi_n$ based on $y_0$ and defined in $[0,1]$
    with
    \[ \inf_n \int_0^1 \si_c(\xi_n, \dot\xi_n) \, \dd s =0 \quad\hbox{and} \quad
  \inf_n   \int_0^1 |\dot\xi_n| \, \dd s > 0;\]
    \item [\bf(ii)] $y \mapsto S_c(y_0,y)$ is solution to $F=c$ in the whole of $\R^M$;
    \item [\bf(iii)] if a function $u$  is a strict critical
    subsolution in a neighborhood of $y_0$, then $u$ cannot be subsolution to $F=c$
in $\R^M$.
\end{itemize}
\end{Proposition}
\smallskip

 Notice that, in contrast with the compact case,
even if the critical value is finite, the Aubry set can be empty for
Hamiltonian defined in $\R^M \times \R^M$. We derive from
Proposition \ref{OhAubry} {\bf (iii)} adapting the same argument of
Lemma \ref{lemHkdue}:

\smallskip

\begin{Corollary}\label{corohAubry} Assume that the Aubry set is empty,
then for any bounded open set $B$ of $\R^M$, there is a critical
subsolution which is strict in $B$.
\end{Corollary}

\smallskip

We record for later use:

\smallskip

\begin{Proposition}\label{postohAubry} Let $B$, $b$  be an open bounded set of $\R^M$,
and a critical value, respectively. Assume that the equation $F =b$
admits a strict subsolution in $B$, and  denote by $w$ a subsolution
of $F=b$ in $\R^M$. Then the Dirichlet  problem
\[\left \{\begin{array}{cc}
   F(y,Du)=  & b \qquad\hbox{in $B$} \\
    u = & w  \qquad\hbox{on $\partial B$}\\
  \end{array} \right . \]
  admits an unique solution $u$ given by the formula
  \[ u(y)= \inf \{ w(z) + S_b(z,y)  \mid z \in \partial B\}. \]
\end{Proposition}

We now consider a supercritical value $b$ and a function $h: \R^M
\to \R$ with
\begin{equation}\label{length0}
h \geq 1 \quad\hbox{in $\R^M$ \, and} \quad h(y) > 1 \Rightarrow
F(y,0)\leq b \, .
\end{equation}
We define for any  curve $\xi$ in $[0,1]$ the length functional
\[\int_0^1 h( \xi) \, \si_b(\xi, \dot\xi) \, ds \]
and denote by $S_b^h$ the corresponding distance obtained as the
infimum of lengths of curves joining two given points of $\R^M$. We
have
\smallskip

\begin{Proposition}\label{length}
Let $b$, $h$  be a supercritical value for $F$ and a function
satisfying \eqref{length0}, respectively, then $S^h_b(z_0,\cdot)$ is
a locally Lipschitz--continuous supersolution to \eqref{eq} in $\R^M
\setminus \{z_0\}$, for any $z_0 \in \R^M$.
\end{Proposition}
\begin{proof} We fix $z_0$. For any $(y,v) \in \R^M \times \R^M$,  $h(y) \,
\si_b(y,v)$ is the support function of the $b$--sublevel of the
Hamiltonian
\begin{equation}\label{length1}
   (y,q) \mapsto F \left ( y, \frac q{h(y)} \right )
\end{equation}
and $S^h_b$ is the corresponding intrinsic distance. According to
Proposition \ref{subsubsub} {\bf (ii)}, $w:= S_b^h(z_0,\cdot)$ is
subsolution to \eqref{eq} in $\R^M$, and supersolution in $\R^M
\setminus \{z_0\}$, with $F$ replaced by the Hamiltonian in
\eqref{length1}. Since the Hamiltonian in \eqref{length1} keeps the
coercivity property of $F$, this implies that $w$ is locally
Lipschitz--continuous in force of Lemma \ref{equili}.

  Taking into
account the supersolution  information on $w$, we consider a
subtangent $\psi$ to $w$ at a point $y$. If  $h(y)=1$ then
\begin{equation}\label{length2}
    F(y,D\psi(y)) = F \left ( y, \frac {D\psi(y)}{h(y)} \right ) \geq
b.
\end{equation}
If instead $h(y) > 1$ then by \eqref{length0} and convex character
of $F$
\begin{eqnarray*}
  F \left ( y, \frac {D\psi(y)}{h(y)} \right ) &=&
 F \left ( y,  \left( 1 - \frac 1{h(y)} \right ) \, 0 +\frac {D\psi(y)}{h(y)} \right ) \\
  &\leq& \frac 1{h(y)} \, F(y,D \psi(y)) + \left( 1 - \frac 1{h(y)} \right
  ) \, b
\end{eqnarray*}
and consequently
\begin{equation}\label{length3}
    \frac 1{h(y)} \, F(y,D \psi(y)) \geq b - \left( 1 - \frac 1{h(y)} \right
  ) \, b = \frac 1{h(y)} \, b.
\end{equation}
 Formulas \eqref{length2}, \eqref{length3}   provide the assertion.
\end{proof}

\end{document}